\documentclass{amsart}
\usepackage{inputenc}
\usepackage{xypic}
\input xy
\xyoption{all}
\usepackage{epsfig}
\usepackage{color}
\usepackage{amsthm}
\usepackage{changepage}
\usepackage{mathdots}
\usepackage{amsmath}
\usepackage{amsfonts}
\usepackage{amssymb}
\usepackage{mathrsfs}
\usepackage{amscd}
\usepackage{amsopn}
\usepackage{enumerate}
\usepackage{mathtools}
\usepackage{url}
\usepackage{graphicx}
\usepackage{hyperref}
\usepackage[usenames]{xcolor}
\usepackage[normalem]{ulem}

\usepackage{hyperref}\hypersetup{colorlinks}

\usepackage{color} 

\definecolor{darkred}{rgb}{1,0,0} 
\definecolor{darkgreen}{rgb}{0,0.8,0}
\definecolor{darkblue}{rgb}{0,0,1}

\hypersetup{colorlinks,
linkcolor=darkblue,
filecolor=darkgreen,
urlcolor=darkred,
citecolor=darkgreen}

 \numberwithin{equation}{section}

\newtheorem {Theorem}{Theorem}

\numberwithin{Theorem}{section}

\newtheorem {Lemma}[Theorem]    {Lemma}

\newtheorem {Proposition}[Theorem]{Proposition}
\newtheorem {Corollary}[Theorem]{Corollary}
\theoremstyle{definition}
\newtheorem{Definition}[Theorem]{Definition}

\newtheorem{Remark}[Theorem]{Remark}
\newtheorem{Example}[Theorem]{Example}

\newtheoremstyle{MyTheorem}
        {.6em}{.6em}              
        {\itshape}                      
        {}                              
        {\bfseries}                     
        {. }                             
        { }                             
        {\thmname{#1}\thmnumber{\addtocounter{MyTheorem}{-4}#2}\thmnote{ \bfseries #3}}
\theoremstyle{MyTheorem}

\newtheoremstyle{TheoremForIntro} 
        {.6em}{.6em}              
        {\itshape}                      
        {}                              
        {\bfseries}                     
        {.}                             
        { }                             
        {\thmname{#1}\thmnote{ \bfseries #3}}
    \theoremstyle{TheoremForIntro}

\expandafter\chardef\csname pre amssym.def at\endcsname=\the\catcode`\@
\catcode`\@=11
\def\undefine#1{\let#1\undefined}
\def\newsymbol#1#2#3#4#5{\let\next@\relax
 \ifnum#2=\@ne\let\next@\msafam@\else
 \ifnum#2=\tw@\let\next@\msbfam@\fi\fi
 \mathchardef#1="#3\next@#4#5}
\def\mathhexbox@#1#2#3{\relax
 \ifmmode\mathpalette{}{\m@th\mathchar"#1#2#3}%
 \else\leavevmode\hbox{$\m@th\mathchar"#1#2#3$}\fi}
\def\hexnumber@#1{\ifcase#1 0\or 1\or 2\or 3\or 4\or 5\or 6\or 7\or 8\or
 9\or A\or B\or C\or D\or E\or F\fi}

\font\teneufm=eufm10
\font\seveneufm=eufm7
\font\fiveeufm=eufm5
\newfam\eufmfam
\textfont\eufmfam=\teneufm
\scriptfont\eufmfam=\seveneufm
\scriptscriptfont\eufmfam=\fiveeufm

\catcode`\@=\csname pre amssym.def at\endcsname

\newcommand{\fsu}{{\mathfrak {su}}}
\newcommand{\fsl}{{\mathfrak {sl}}}

\newcommand{\sSL}{{\mathsf {SL}}}

\newcommand{\sU}{{\mathsf U}}

\newcommand{\Hod}{{\mathrm {Hod}}}
\newcommand{\VHS}{{\mathsf{VHS}}}

\newcommand{\NHC}{{\mathsf{NHC}}}
\newcommand{\BB}{{\mathsf{BB}}}
\newcommand{\BAA}{{\mathsf{BAA}}}
\newcommand{\ABA}{{\mathsf{ABA}}}

\newcommand{\Hom}{{\mathrm{Hom}}}

\newcommand{\End}{{\mathrm{End}}}
\newcommand{\ad}{{\mathrm{ad}}}

\newcommand{\rdR}{{\mathrm{dR}}}
\newcommand{\rL}{{\mathrm{L}}}
\newcommand{\rH}{{\mathrm{H}}}
\newcommand{\rW}{{\mathrm{W}}}
\newcommand{\Tr}{{\mathrm{Tr}}}

\newcommand{\Ee}{{\mathcal E}}
\newcommand{\Ff}{{\mathcal F}}

\newcommand{\Gg}{{\mathcal G}}
\newcommand{\Hh}{{\mathcal H}}
\newcommand{\Kk}{{\mathcal K}}
\newcommand{\Mm}{{\mathcal M}}

\newcommand{\Oo}{{\mathcal O}}

\newcommand{\Uu}{{\mathcal U}}

\newcommand{\Ss}{{\mathcal S}}
\newcommand{\Tt}{{\mathcal T}}

\newcommand{\mtrx}[1]{\left (\begin{matrix}#1\end{matrix}\right)}

\def    \C      {{\mathbb C}}
\def    \R      {{\mathbb R}}
\def    \Z      {{\mathbb Z}}

\def    \lra     {{\longrightarrow}}
\def    \haf    {{\frac{1}{2}}}

\def    \p      {\partial}

\def    \rk     {\operatorname{rk}}

\newcommand{\An}{\xymatrix{ *{\circ} \ar@{-}[r]|*\dir{ } & *{\circ}\ar@{-}[r]&{\cdots}&*{\circ}\ar@{-}[l]|*\dir{ }\ar@{-}[r]|*\dir{ }&*{\circ} }}
\newcommand{\Anlabel}{\xymatrix@R=.25em{ *{\circ}\ar@<-1ex>@{}[d]^{\alpha_{1}} \ar@{-}[r]|*\dir{ } & *{\circ}\ar@<-1ex>@{}[d]^{\alpha_{2}}\ar@{-}[r]&{\cdots}&*{\circ}\ar@{-}[l]|*\dir{ }\ar@{-}[r]|*\dir{ }\ar@<-2ex>@{}[d]^{\alpha_{n-1}}&*{\circ}\ar@<-1ex>@{}[d]^{\alpha_{n}}\\&&&& }}
\newcommand{\Bn}{\xymatrix{ *{\circ} \ar@{-}[r]|*\dir{ } & *{\circ}\ar@{-}[r]&{\cdots}&*{\circ}\ar@{-}[l]|*\dir{ }\ar@{=}[r]|*\dir{>}&*{\circ} }}
\newcommand{\Bnlabel}{\xymatrix@R=.25em{ *{\circ}\ar@<-1ex>@{}[d]^{\alpha_{1}} \ar@{-}[r]|*\dir{ } & *{\circ}\ar@<-1ex>@{}[d]^{\alpha_{2}}\ar@{-}[r]&{\cdots}&*{\circ}\ar@{-}[l]|*\dir{ }\ar@{=}[r]|*\dir{>}\ar@<-2ex>@{}[d]^{\alpha_{n-1}}&*{\circ}\ar@<-1ex>@{}[d]^{\alpha_{n}}\\&&&& }}
\newcommand{\Cn}{\xymatrix{ *{\circ} \ar@{-}[r]|*\dir{ } & *{\circ}\ar@{-}[r]&{\cdots}&*{\circ}\ar@{-}[l]|*\dir{ }\ar@{=}[r]|*\dir{<}&*{\circ} }}
\newcommand{\Cnlabel}{\xymatrix@R=.25em{ *{\circ}\ar@<-1ex>@{}[d]^{\alpha_{1}} \ar@{-}[r]|*\dir{ } & *{\circ}\ar@<-1ex>@{}[d]^{\alpha_{2}}\ar@{-}[r]&{\cdots}&*{\circ}\ar@{-}[l]|*\dir{ }\ar@{=}[r]|*\dir{<}\ar@<-2ex>@{}[d]^{\alpha_{n-1}}&*{\circ}\ar@<-1ex>@{}[d]^{\alpha_{n}}\\&&&& }}
\newcommand{\Dn}{\xymatrix@R=.25em{&&&&*{\circ} \\ *{\circ} \ar@{-}[r]|*\dir{ } & *{\circ}\ar@{-}[r]&{\cdots}&*{\circ}\ar@{-}[l]|*\dir{ }\ar@{-}[ur]|*\dir{ }\ar@{-}[dr]|*\dir{ }& \\ &&&&*{\circ} }}
\newcommand{\Dnlabel}{\xymatrix@R=.25em{&&&&*{\circ}\ar@<-1ex>@{}[d]^{\alpha_{n-1}} \\ *{\circ}\ar@<-1ex>@{}[d]^{\alpha_{1}} \ar@{-}[r]|*\dir{ } & *{\circ}\ar@<-1ex>@{}[d]^{\alpha_{2}}\ar@{-}[r]&{\cdots}&*{\circ}\ar@{-}[l]|*\dir{ }\ar@{-}[ur]|*\dir{ }\ar@{-}[dr]|*\dir{ }\ar@<-4ex>@{}[d]^{\alpha_{n-2}}& \\ &&&&*{\circ}\ar@<-2ex>@{}[u]^(-.5){\alpha_{n}} }}

\newcommand{\Ibold}{{\bf I}}

\newcommand{\ubold}{{\bf u}}

\newcommand{\CBbb}{\mathbb C}

\newcommand{\PBbb}{\mathbb P}

\newcommand{\Ecal}{\mathcal E}
\newcommand{\Fcal}{\mathcal F}
\newcommand{\Gcal}{\mathcal G}
\newcommand{\Hcal}{\mathcal H}

\newcommand{\Kcal}{\mathcal K}

\newcommand{\Ocal}{\mathcal O}

\newcommand{\Scal}{\mathcal S}
\newcommand{\Tcal}{\mathcal T}

\newcommand{\Ascr}{\mathscr A}

\newcommand{\Dscr}{\mathscr D}

\newcommand{\Hscr}{\mathscr H}

\newcommand{\slfrak}{\mathfrak{sl}}

\newcommand{\sufrak}{\mathfrak{su}}

\DeclareMathOperator{\Met}{Met}

\DeclareMathOperator{\rank}{rank}

\DeclareMathOperator{\Gr}{Gr}
\DeclareMathOperator{\ind}{index }

\newcommand{\sbt}{\,\begin{picture}(-1,1)(-1,-1)\circle*{2}\end{picture}\ }
\newcommand*{\dt}[1]{\overset{\sbt}{#1}}

\newcommand{\dbar}{\bar\partial}

\newcommand{\isorightarrow}{\xrightarrow{
   \,\smash{\raisebox{-0.5ex}{\ensuremath{\sim}}}\,}}

\usepackage{scalerel,stackengine}
\stackMath
\newcommand\widefrown[1]{%
\savestack{\tmpbox}{\stretchto{%
  \scaleto{%
    \scalerel*[\widthof{\ensuremath{#1}}]{\kern-.6pt\frown\kern-.6pt}%
    {\rule[-\textheight/2]{1ex}{\textheight}}
  }{\textheight}%
}{0.5ex}}%
\stackon[1pt]{#1}{\tmpbox}%
}

\theoremstyle{remark}

\author{Brian Collier}
\email{briancollier01@gmail.com}

\author{Richard Wentworth}
\email{raw@umd.edu}

\address{Department of Mathematics,
   University of Maryland,
   College Park, MD 20742, USA}
   
\thanks{B.C. is supported by the NSF under Award No. 1604263. R.W. was supported in part by NSF grants DMS-1406513 and DMS-1564373. 
 The authors also acknowledge support from NSF grants DMS-1107452, -1107263, -1107367 ``RNMS: GEometric structures And Representation varieties'' (the GEAR Network).}

\title[Conformal limits and $\lambda$-connections]{Conformal limits and the Bia{\l}ynicki-Birula stratification of the space of $\lambda$-connections}

\begin{document}
\begin{abstract}
The Bia{\l}ynicki-Birula decomposition of the space of $\lambda$-connections restricts to the Morse stratification 
on the moduli space of Higgs bundles and to the partial oper stratification on the de Rham moduli space of holomorphic connections.  For both the Morse and partial oper stratifications, every stratum is a holomorphic Lagrangian fibration over a component of the space of complex variations of Hodge structure. In this paper, we generalize known results for the Hitchin section and the space of opers to arbitrary strata. These include the following:
a biholomorphic identification of the fibers of the two strata over a stable variation of Hodge structure  via the ``$\hbar$-conformal limit'' of Gaiotto, a proof that the  fibers of the Morse and partial oper stratifications are transverse at the base point, and an explicit parametrization of the fibers as half-dimensional affine spaces,
\end{abstract}

\date{\today}

\maketitle
\setlength{\smallskipamount}{6pt}
\setlength{\medskipamount}{10pt}
\setlength{\bigskipamount}{16pt}

\section{Introduction}

For a compact Riemann surface $X$, both Hitchin's moduli space of Higgs bundles and the de Rham moduli space of
holomorphic connections are interesting holomorphic symplectic  manifolds. While the celebrated nonabelian Hodge correspondence provides a homeomorphism between the two, they are in fact quite different as complex analytic spaces. 
For example, the Higgs moduli space contains positive dimensional subvarieties (e.g.\ the moduli space of holomorphic bundles), whereas the de Rham moduli space is Stein. Nevertheless, the complex structures are related in the sense that their composition gives a third  complex structure which in fact defines a hyperk\"ahler structure. By a construction of Deligne \cite{DeligneLetter},  further elucidated by Simpson in \cite{SimpsonHodgeFiltration}, the associated twistor space may be interpreted as a moduli space of \emph{$\lambda$-connections}. A further key property of the twistor space is the existence of a $\CBbb^\ast$-action extending the one defined by Hitchin on the Higgs moduli space.

The Higgs and de Rham moduli spaces each have  distinguished holomorphic Lagrangian submanifolds which can be parametrized by a particular vector space of holomorphic differentials on $X$.    On the Higgs bundle side, this space was defined by Hitchin in \cite{liegroupsteichmuller} as the image of a section of the Hitchin fibration, and it generalizes the harmonic maps parametrization  of Teichm\"uller space (cf.\ \cite{TeichOfHarmonic}). 
On the de Rham side, this is the space of \emph{opers},  which may be considered as a generalization of complex projective structures on $X$ (cf.\  \cite{BeilinsonDrinfeldOPERS}). 

In fact, the Hitchin section and the set of opers  each form a closed stratum of natural stratifications on the spaces of Higgs bundles and holomorphic connections coming from an application of Bia{\l}ynicki-Birula ($\BB$) theory for the $\CBbb^\ast$-action mentioned above \cite{SimpsonDeRhamStrata}. 
In both cases, each stratum is a holomorphic fibration by Lagrangians over a connected component of the set of complex variations of Hodge structure.  In the language introduced by Kapustin and Witten \cite{KapustinWitten}, the fibers are submanifolds  of type ($\BAA$) and ($\ABA$), respectively.

The goal of this paper is to extend the many relationships between the Hitchin section and the space of opers to all components of the $\BB$-stratification on the space of $\lambda$-connections.
The main results are:
\begin{enumerate}
\item a proof of the existence of an ``$\hbar$-conformal limit'' (in the sense of Gaiotto \cite{GaiottoConj}) for every Higgs bundle whose $\C^*$-limit is stable; 
\item  a proof that the de Rham fiber is transverse to the Higgs bundle fiber at each stable complex variation of Hodge structure;
\item
 a global parametrization of each of the fibers in the $\BB$-stratification as a half dimensional affine space and a proof that their intersections with the Higgs and de Rham moduli spaces are submanifolds of type ($\BAA$) and ($\ABA$), respectively.
\end{enumerate}
The following subsections contain more precise statements of these items.

\subsection{Conformal limits}
Let $E\to X$ be a complex vector bundle of rank $n$ and degree zero.
 We denote by $\Mm_{\rH}$, $\Mm_{\rdR}$, and $\Mm_{\Hod}$ the moduli spaces of  Higgs bundles, holomorphic connections, and $\lambda$-connections, respectively, on the underlying smooth bundle $E$ (see Section \ref{sec:moduli} for definitions).
There is a holomorphic map $\lambda    : \Mm_\Hod\to \CBbb$ with
\[\xymatrix{\Mm_\rH = \lambda^{-1}(0)&\text{and}&\Mm_\rdR=\lambda^{-1}(1)~.}\]

In \cite{SimpsonDeRhamStrata}, Simpson showed that the $\lambda$-equivariant $\C^*$-action 
on $\Mm_{\Hod}$ given by
 $\nabla_\lambda\mapsto \xi\cdot \nabla_\lambda$, $\xi\in \C^\ast$,
defines a Bia{\l}ynicki-Birula type decomposition. 
The strata $\rW_\alpha\subset\Mm_{\Hod}$ are labeled by the components $\VHS_\alpha\subset\Mm_{\rH}$ of the fixed point set. These are the Higgs bundles defining \emph{complex variations of Hodge structure} ($\VHS$). On each stratum there is a holomorphic map $\rW_\alpha\to \VHS_\alpha$ with affine fibers.  
The restrictions $\rW_\alpha^0$, $\rW_\alpha^1$ to $\lambda=0, 1$ are respectively the \emph{Morse strata} of $\Mm_{\rH}$ defined by Hitchin and the \emph{partial oper strata} of $\Mm_{\rdR}$ defined by Simpson (see Section \ref{sec:strata}).

A choice of hermitian metric $h$ on $E$ associates hermitian adjoints $\ast_h$ to endomorphisms and Chern connections 
$\dbar_E+\partial_E^h$ to $\dbar$-operators. The nonabelian Hodge correspondence ($\NHC$) produces  a special hermitian metric, known as the \emph{harmonic metric}, to both a polystable Higgs bundle and a completely reducible flat connection. Through this association, the $\NHC$ defines a homeomorphism:
\begin{equation}
    \label{eq NAH intro}
    \Tt: \Mm_\rH \isorightarrow \Mm_\rdR : [(\dbar_E, \Phi)] \mapsto [D=\dbar_E+\partial_E^h+\Phi+\Phi^\ast]\ 
\end{equation}
between Higgs bundles and flat connections  
(see Section \ref{sec:nhc} for more details).

For  a stable Higgs bundle $(\dbar_E,\Phi)$ with harmonic metric $h$, there are two ways to scale the correspondence in \eqref{eq NAH intro}
 -- one uses the $\CBbb^\ast$-action on $\Mm_{\rH}$ and the other a so-called \emph{real twistor line} in $\Mm_{\Hod}$. More precisely, 
if $R>0$, then $(\dbar_E, R\Phi)$ is also stable, and so we obtain harmonic metrics $h_R$. Next, for $\xi\in \CBbb^\ast$, the connection
$$
\dbar_E+\partial_E^{h_R}+\xi^{-1}R\Phi+\xi R\Phi^{\ast_{h_R}}
$$
is also flat. If we keep $\hbar=R^{-1}\xi$ fixed, we obtain a family of flat connections parametrized by $R$:
\begin{equation} \label{eqn:gaiotto}
D_{R,\hbar}=D_{R,\hbar}(\dbar_E,\Phi)=\dbar_E+\partial_E^{h_R}+\hbar^{-1}\Phi+\hbar R^2\Phi^{\ast_{h_R}}\ .
\end{equation}
The limiting flat connection $\lim_{R\to 0} D_{R,\hbar}$, 
when it exists, is called the \emph{$\hbar$-conformal limit} of the Higgs bundle $(\dbar_E,\Phi)$. The existence of a limit is a gauge invariant property.
For example, if $\hbar=1$ and $(\dbar_E,\Phi)$ is a fixed point of the $\CBbb^\ast$-action, then the conformal limit is just the associated flat connection from the $\NHC$.
More generally, in \cite{GaiottoConj} Gaiotto conjectured that for $[(\dbar_E,\Phi)]$ in the Hitchin component, the 
$\hbar$-conformal limit exists and lies in the space of opers, and moreover this gives a biholomorphism between these spaces. This conjecture was recently verified in \cite{GaiottoLimitsOPERS}.

In this paper we generalize the conformal limit correspondence.  
The first main result is the following. 
\begin{Theorem}[{\sc conformal limit}] \label{thm:ConformalLimit}
Let $[(\dbar_E,\Phi)]\in \Mm_\rH$ be such that its limit $[(\dbar_0,\Phi_0)]\in \VHS_\alpha$ under the $\CBbb^\ast$-action is represented by a stable Higgs bundle. Then 
 the $\hbar$-conformal limit of $[(\dbar_E,\Phi)]$ exists.  Moreover, the conformal limit gives a biholomorphism between the fibers $\rW^0_\alpha(\dbar_0,\Phi_0)$ and
 $\rW^1_\alpha(\dbar_0,\Phi_0)$ of $\rW_\alpha^0$ and $\rW_\alpha^1$ over $[(\dbar_0,\Phi_0)]$.
\end{Theorem}

\subsection{Transversality of fibers at the $\VHS$}
 As mentioned above the moduli spaces $\Mm_\rH$ and $\Mm_\rdR$ are holomorphic symplectic manifolds. Let $(I,\omega_\C^I)$ and $(J,\omega_\C^J)$ denote the complex structures and holomorphic symplectic forms on $\Mm_\rH$ and $\Mm_\rdR$, respectively.  These are not preserved by the homeomorphism \eqref{eq NAH intro}, and their composition $K=I\circ J$ defines a third complex structure, giving $\Mm_\rH$ the structure of a hyperk\"ahler manifold. 
\begin{Proposition} \label{prop:lagrange}
Let $[(\bar\p_E,\Phi)]\in\VHS_\alpha$ be stable.  Then the fiber $\rW_\alpha^0(\dbar_E,\Phi)$ of $\rW_\alpha^0$ over $[(\bar\p_E,\Phi)]$ is an $\omega^I_\C$ holomorphic Lagrangian submanifold of $\Mm_{\rH}$ $($type $($$\BAA$$)$$)$, and $\rW_\alpha^1(\bar\p_E,\Phi)$ is an $\omega^J_\C$ holomorphic Lagrangian submanifold of $\Mm_{\rdR}$ $($type $($$\ABA$$)$$)$.\footnote{The second statement in Proposition \ref{prop:lagrange} was already proved in \cite[Lemma 7.3]{SimpsonDeRhamStrata}.}
\end{Proposition}
In particular, the tangent space splits as 
 \[T_{\Tt(\bar\p_E,\Phi)}\Mm_{\rdR}=T_{\Tt(\bar\p_E,\Phi)}\rW_\alpha^1(\bar\p_E,\Phi)\oplus K\big(T_{\Tt(\bar\p_E,\Phi)}\rW_\alpha^1(\bar\p_E,\Phi)\big)~.\]
Moreover,  $\Tt(\rW^0_\alpha(\bar\p_E,\Phi))$ intersects $\rW^1_\alpha(\bar\p_E,\Phi)$ at $\Tt(\bar\p_E,\Phi).$ We shall prove  that this intersection is transverse, generalizing
the known result for the Hitchin component and opers (cf. \cite{LabourieWentworthFuchLocus}).
\begin{Theorem}[{\sc Transversality at the $\VHS$}] \label{thm:transverse}
      For each stable $(\bar\p_E,\Phi)\in\VHS_\alpha$, the transformation  
    $\mu\mapsto \mu- K(\mu)$
    maps the tangent space $T_{\Tt(\bar\p_E,\Phi)}\rW^1_\alpha(\bar\p_E,\Phi)$ linearly isomorphically onto $d\Tt(T_{(\bar\p_E,\Phi)}\rW^0_\alpha(\bar\p_E,\Phi))$. In particular, $$\Tt(\rW^0_\alpha(\bar\p_E,\Phi))\pitchfork\rW^1_\alpha(\bar\p_E,\Phi)$$ at $\Tt(\bar\p_E,\Phi).$
\end{Theorem}

\subsection{Global slices}
Theorems \ref{thm:ConformalLimit} and \ref{thm:transverse} are consequences of the existence of a  \emph{global slice} in the deformation theory about a $\VHS$. More precisely, 
for a stable Higgs bundle $(\dbar_E,\Phi)$, the $\NHC$ allows us to define a \emph{Hodge slice} $\Scal_{(\dbar_E,\Phi)}$ in the space of Higgs bundles
(see Definition \ref{def:slice}). 
Via the Kuranishi map, this
  gives a local coordinate chart of $\Mm_{\rH}$ about the point $[(\dbar_E,\Phi)]$.  
If furthermore $[(\dbar_E,\Phi)]\in \VHS_\alpha$, then there is a holomorphic splitting of $(E,\dbar_E)$. 
Since the deformation theory of $\Mm_{\rH}$ is compatible with this splitting, this extra structure defines a natural subvariety $\Scal^+_{(\dbar_E,\Phi)}\subset \Scal_{(\dbar_E,\Phi)}$ 
which we call the \emph{Bia{\l}ynicki-Birula} \emph{slice} (see Sections \ref{sec:hodge-slice} and \ref{sec:bb-slice}). 

We will show that the Bia{\l}ynicki-Birula slice $\Ss^+_{(\dbar_E,\Phi)}$ is biholomorphic to the fibers $\rW_\alpha^0(\bar\p_E,\Phi)$ and $\rW_\alpha^1(\bar\p_E,\Phi),$ generalizing the parameterizations of the Hitchin section and opers the Hitchin base (see Section \ref{sec: HitchinSection and Opers}).
\begin{Theorem}[{\sc Global slice}] \label{thm:global-slice}
Let $[(\bar\p_E,\Phi)]\in\VHS_\alpha$ be a stable Higgs bundle with $\BB$-slice $\Ss^+_{(\bar\p_E,\Phi)}$,
and let $\rW_\alpha(\dbar_E,\Phi)$ $($resp.\ $\rW_\alpha^\xi(\dbar_E,\Phi)$$)$ denote the fiber of $\rW_\alpha \to \VHS_\alpha$ $($resp.\ $\rW_\alpha\cap\lambda^{-1}(\xi)\to \VHS_\alpha$$)$ over $[(\bar\p_E,\Phi)]$. 

 Then:
\begin{enumerate}
\item the Kuranishi map gives a biholomorphism of $\Ss^+_{(\bar\p_E,\Phi)}$ with a complex vector space of half the dimension of $\Mm_{\rH}$;
\item the natural projection $p_{\rH}: \Ss^+_{(\bar\p_E,\Phi)}\to \Mm_{\rH}$ is a biholomorphic embedding onto the Morse fiber $\rW_\alpha^0(\dbar_E,\Phi)$;
\item 
there is a holomorphic  map $p_{\Hod}: \Scal^+_{(\dbar_E,\Phi)}\times\CBbb \to  \Mm_{\Hod} $, extending $p_{\rH}$ and
making the following diagram commute
\begin{equation*}
\begin{split}
\xymatrix{
\Scal^+_{(\dbar_E,\Phi)}\times\CBbb  \ar[dr]_{{\rm pr}_2} \ar[rr]^{p_{\Hod}} && \Mm_{\Hod} \ar[dl]^{\lambda} \\
&\CBbb &
}
\end{split}
\end{equation*}
where ${\rm pr}_2$ denotes projection to the second factor.
Moreover, $p_{\Hod}$ is a biholomorphic embedding onto the $\BB$-fiber $\rW_\alpha(\dbar_E,\Phi)$.   In particular, the fibers $\rW^0_\alpha{(\dbar_E,\Phi)}$ and $\rW^1_\alpha{(\dbar_E,\Phi)}$ are biholomorphic.
\end{enumerate}
\end{Theorem} 

By standard $\BB$-theory, the fibers $\rW_\alpha^0(\dbar_E,\Phi)\subset \Mm_{\rH}$ of the Morse strata are affine spaces. 
The theorem above immediately shows that this  property generalizes to the fibers of the partial oper stratification in $\Mm_\rdR$. 
\begin{Corollary}
For each stable $[(\bar\p_E,\Phi)]\in\VHS_\alpha$ the fiber $\rW^1_\alpha{(\dbar_E,\Phi)}\subset\Mm_\rdR$ of $\rW_\alpha \to \VHS_\alpha$ over $[(\bar\p_E,\Phi)]$ is affine.
\end{Corollary}
In fact,  the biholomorphism $\rW_\alpha^0(\bar\p_0,\Phi_0)\simeq \rW_\alpha^1(\bar\p_0,\Phi_0)$ in Theorem \ref{thm:global-slice} (3) comes from the conformal limit.
\begin{Corollary} \label{cor:ConformalLimit}
    Let $p_\hbar:\Ss^+_{(\bar\p_E,\Phi)}\to \rW^\hbar_\alpha(\bar\p_E,\Phi)$ denote the biholomorphism obtained by restricting  $p_{\Hod}$ from Theorem \ref{thm:global-slice} to $\Ss^+_{(\bar\p_E,\Phi)}\times\{\hbar\}$.  Then the $\hbar$-conformal limit through the Higgs bundle $p_0(\ubold)$ is $\hbar^{-1}\cdot p_\hbar(\ubold)$. 
 \end{Corollary}
Hence, the process of taking conformal limits may be regarded as a deformation of $\Mm_{\rdR}$ which moves the Morse stratification to the partial oper stratification, preserving the $\VHS$ locus.
\begin{Remark}
The comparison of the stratifications given by $W^0_\alpha$ and $W^1_\alpha$ is an interesting question. For example, 
the existence of Higgs bundles with nilpotent Higgs field implies that the fibers $W^0_\alpha(\dbar_E,\Phi)$ are not, in general, closed. On the other hand, as pointed out in \cite{SimpsonDeRhamStrata}, it is unknown  whether the fibers $W^1_\alpha(\dbar_E,\Phi)$ are themselves closed (see also Remark \ref{rmk:discontinuous} below). In \cite{SimpsonDeRhamStrata} and \cite{ClosureOfDeRhamStrata4P1}, the structure of the closure of the strata $W^1_\alpha$ is described for rank 2. 
The general case remains an open problem. 
\end{Remark}

\noindent\textbf{Acknowledgments:} It is a pleasure to thank Andy Neitzke and Carlos Simpson for enlightening conversations.


\section{Moduli spaces} \label{sec:moduli}

\subsection{The de Rham moduli space of flat connections}
As in the Introduction, we
fix a smooth rank $n$ complex vector bundle $E\to X$ with a trivialization of the determinant line bundle $\det E$. Let $\Dscr=\Dscr(E,\sSL(n,\C))$ denote the space of flat connections that induce the trivial connection on $\det E$. The space of all such connections is an infinite dimensional affine space modeled on the vector space $\Omega^1(X,\fsl(E))$,\footnote{Since $X$ is fixed throughout, we will henceforth omit it from the notation and write $\Omega^i(\ast)$ instead of $\Omega^i(X,\ast)$.} 
where $\fsl(E)\subset\End(E)$ denotes the subbundle of traceless endomorphisms. 
In particular, the tangent space to $\Dscr$ at a flat connection $D$ is the set of $\mu\in\Omega^1(\fsl(E))$ so that $D+\mu$ is flat to first order: 
\[T_D\Dscr=\{\mu\in\Omega^1(\fsl(E))~|~D\mu=0\}~.\] 

The complex structure on $\sSL(n,\C)$ defines a complex structure on $\Dscr$; namely, 
\begin{equation*}
    \label{eq J complex structure on de Rham}J(\mu)=i\mu~.
\end{equation*}
Moreover, the alternating form 
\begin{equation}\label{eq: ABG hol symplectic form on de Rham}
    \omega_J^\C(\mu,\nu)=\int\limits_X\Tr(\mu\wedge\nu)
\end{equation}
defines a ($J$-)holomorphic symplectic form on $\Dscr.$ This is the Atiyah-Bott-Goldman-Narasimhan symplectic form. Note that both $J$ and $\omega_J^\C$ are independent of the Riemann surface structure $X$. 

The gauge group $\Gg(E)$ of smooth  automorphisms of $E$ that act trivially  on $\det E$ acts on the space of flat connections by $D\cdot g=g^{-1}\circ D\circ g$. 
A connection $D$ is called \emph{completely reducible} if every $D$-invariant subbundle $F\subset E$ has an $D$-invariant complement,  and $D$ is called \emph{irreducible} if there are 
no nontrivial $D$-invariant subbundles. The orbits of the gauge group action on the 
set of completely reducible flat connections are closed. We define the \emph{de Rham moduli space} of flat connections by
\[\Mm_{\rdR}=\Dscr^{\mathrm{cr}}/\Gg(E)~,\]
where $\Dscr^{\mathrm{cr}}\subset\Dscr$ denotes the set of completely reducible flat connections.

The complex structure $J$ and the symplectic form $\omega_J^\C$ are clearly preserved by the gauge group action. In fact,  $J$ gives $\Mm_{\rdR}$ the structure of a (singular) complex analytic space which is Stein and has complex dimension $(n^2-1)(2g-2)$  (see  \cite{Simpson:94b}, Props.\ 6.1, 7.8, and Cor.\ 11.7). The smooth locus of $\Mm_{\rdR}$ consists of equivalence classes of irreducible representations and the symplectic form $\omega^\C_J$ defines a holomorphic symplectic form on the smooth locus.

\subsection{The Higgs bundle moduli space}
We now describe the Higgs bundle moduli space and its holomorphic symplectic structure.
Throughout this paper we identify the set of holomorphic structures on $E\to X$  with the set $\Ascr=\Ascr(E)$ of $\dbar$-operators $\bar\p_E:\Omega^{0}(E)\to\Omega^{0,1}(E).$ Holomorphic bundles (or their associated coherent sheaves) will sometimes be denoted by $\Ecal$, and when we wish to emphasize this correspondence we will write $\Ecal=(E,\dbar_E)$. 
Recall that we always assume that $\det E$ has a fixed trivialization and that $\dbar$-operators induce on $\det E$ the $\dbar$-operator compatible with the trivialization.  In other notation, there is a fixed isomorphism $\Lambda^n\Ee\cong\Ocal_X$.

Let $K\to X$ denote the canonical line bundle of $X$.
\begin{Definition}\label{Def: Higgs bundle}
    An $\sSL(n,\C)$-Higgs bundle on $X$ is a pair $(\dbar_E,\Phi)$ where
    $\Phi:E\to E\otimes K$ is holomorphic: $\dbar_E\Phi=0$, and traceless: $\Tr(\Phi)=0$.
    \end{Definition} 
\noindent
The holomorphic section $\Phi$ is called the \emph{Higgs field}. 
  Let $\Hscr=\Hscr(E,\sSL(n,\C))$ denote the set of $\sSL(n,\C)$-Higgs bundles  on  $E$. 
Since $\Ascr$ is an affine space modeled on the vector space $\Omega^{0,1}(\fsl(E))$, 
the tangent space to $\Hscr$ at a Higgs bundle $(\bar\p_E,\Phi)$ is the set of $(\beta,\varphi)\in\Omega^{0,1}(\fsl(E))\oplus\Omega^{1,0}(\fsl(E))$ such that, to first order, $\Phi+\varphi$ is holomorphic with respect to $\bar\p_E+\beta.$ That is, 
\[T_{(\bar\p_E,\Phi)}\Hscr=\{(\beta,\varphi)\in\Omega^{0,1}(\fsl(E))\oplus\Omega^{1,0}(\fsl(E))~|~ \bar\p_E\varphi+[\Phi,\beta]=0\}~.\]
The space $\Hscr$ carries a natural complex structure $I$ defined by
\begin{equation*}\label{eq I complex structure Higgs}
    I(\beta,\varphi)=(i\beta,i\varphi)~.
\end{equation*}
Moreover, the alternating form 
\begin{equation}
    \label{eq hol symplectic form on Higgs bundles}\omega_I^\C((\beta_1,\varphi_1),(\beta_2,\varphi_2))=i\int\limits_X\Tr(\varphi_2\wedge\beta_1-\varphi_1\wedge\beta_2)
\end{equation}
defines a (I-)holomorphic symplectic form on $\Hscr.$ 

The gauge group $\Gg(E)$ acts on  $\Hscr$ by the adjoint action: 
\[(\bar\p_E,\Phi)\cdot g=(g^{-1}\circ\bar\p_E\circ g~,~ g^{-1}\circ\Phi\circ g)~.\]
To form the relevant moduli space of Higgs bundles we need the notion of stability. 
\begin{Definition} \label{def:stable}
  A Higgs bundle $(\dbar_E,\Phi)$,  $\Ecal=(E,\dbar_E)$,  is called 
    \begin{itemize}
        \item \emph{semistable} if for all $\Phi$-invariant subbundles $\Ff\subset\Ee$ we have $\deg(\Ff)\leq0,$
        \item \emph{stable} if for all $\Phi$-invariant subbundles $\Ff\subset\Ee$, $0<\rank(\Fcal)<\rank(\Ecal)$, we have $\deg(\Ff)<0$, and
        \item \emph{polystable} if it is semistable and a direct sum of Higgs bundles $(\Ee_j,\Phi_j)$ so that $\deg(\Ee_j)=0$ for all $j$ and each $(\Ee_j,\Phi_j)$ is stable.   
         \end{itemize} 
\end{Definition}
The orbits of the gauge group action on the set of polystable Higgs bundles are closed. We define the \emph{Higgs bundle moduli space} by
\[\Mm_{\rH}=\Hscr^{\mathrm{ps}}/\Gg(E)~,\]
where $\Hscr^{\mathrm{ps}}\subset\Hscr$ is the set of polystable Higgs bundles.

The complex structure $I$ and the symplectic form $\omega_I^\C$ are clearly preserved by the gauge group action. In fact, the complex structure $I$ gives $\Mm_{\rH}$ the structure of a (singular) complex analytic space which is a normal, quasiprojective variety of dimension $(n^2-1)(2g-2)$ \cite{selfduality,NitsureModuliofPairs,Simpson:94b}. The smooth locus of $\Mm_{\rH}$ consists of equivalence classes of stable Higgs bundles and the symplectic form $\omega^\C_I$ defines a holomorphic symplectic form on the smooth locus. 

Another key feature of the Higgs bundle moduli space is the Hitchin fibration
\begin{equation} \label{eqn:hitchin-fibration}
\Mm_\rH\lra\bigoplus\limits_{j=1}^{n-1}H^0(K^{j+1})~:~[(\bar\p_E,\Phi)]\mapsto (p_1(\Phi),\cdots, p_{n-1}(\Phi))~.
\end{equation}
 Here $p_1,\cdots,p_{n-1}$ is a basis of the invariant polynomials $\C[\fsl(n,\C)]^{\sSL(n,\C)}$ with $\deg(p_j)=j+1.$
In \cite{IntSystemFibration}, Hitchin showed that the above map is proper and gives $\Mm_\rH$ the structure of a complex integrable system. 
The feature of the above fibration relevant us is the existence of a section $\bigoplus\limits_{j=1}^{n-1}H^0(K^{j+1})\to\Mm_\rH$ of the fibration above which is called the \emph{Hitchin section} (see Section \ref{sec: HitchinSection and Opers}). Its image  defines a half dimensional $\omega_I^\C$-Lagrangian submanifold of $\Mm_\rH$ (cf. \cite{liegroupsteichmuller}).

\subsection{Nonabelian Hodge correspondence} \label{sec:nhc}
For reference, we state here the fundamental relationship between Higgs bundles and flat connections mentioned
 in the Introduction and used throughout the paper. It
  was established by Hitchin \cite{selfduality} and Donaldson \cite{harmoicmetric} when $\rk(E)=2$, and Simpson \cite{SimpsonVHS} and Corlette \cite{canonicalmetrics} in general. 
\begin{Theorem}[{$\NHC$}]
    \label{thm:nhc} 
    A Higgs bundle $(\bar\p_E,\Phi)$  is polystable if and only if there exists a hermitian metric $h$ on $E$ such that 
 \begin{equation} \label{eqn:hitchin}
F_{(\dbar_E,h)}+[\Phi,\Phi^{\ast_h}]=0\ ,
\end{equation}
    where $F_{(\dbar_E,h)}$ is the curvature of the Chern connection $\dbar_E+\partial_E^h$ associated to the pair $(\dbar_E, h)$,  and $\Phi^{*_h}$ is the $h$-adjoint of $\Phi$. Moreover, the connection $D$ defined by
\begin{equation} \label{eqn:corlette}
    D=\dbar_E+\partial_E^h+\Phi+\Phi^{\ast_h}\ ,
    \end{equation}
 is a completely reducible flat connection which is irreducible if and only if $(\bar\p_E,\Phi)$ is stable. 
    Conversely,  a flat connection $D$ is completely reducible  if and only if there exists a hermitian metric $h$ on $E$ so that
    when we express $D$ (uniquely) in the form \eqref{eqn:corlette},
    then $\dbar_E\Phi=0$.  In this case, $(\dbar_E,\Phi)$  is a polystable Higgs bundle, and it is stable  if and only if $D$ is irreducible.
\end{Theorem}

    As we have already done in the Introduction, the hermitian metric $h$ solving 
     \eqref{eqn:hitchin} is referred to as the \emph{harmonic metric} associated to the Higgs bundle $(\bar\p_E,\Phi)$.

\subsection{The hyperk\"ahler structure}

Given a hermitian metric $h$ on the bundle $E$, the Chern connection identifies the space of $\bar\p$-operators with the space of unitary connections $\Ascr=\Ascr(E,h)$ on E. Using this identification, both the space $\Dscr$ of flat connections and the space $\Hscr$ of Higgs bundles may be viewed as subsets of the total space of the cotangent bundle $T^\ast\Ascr$. Since $\Ascr$ is K\"ahler and affine, $T^\ast\Ascr$ has a natural hyperk\"ahler structure, and $\Dscr$ and $\Hscr$ are zeros of (different) complex moment maps for actions by the group $\Kk=\Kk(E,h)$ of unitary gauge transformations.

To be more precise, we may view $T^\ast\Ascr$ as the set of pairs $(d_E, \Psi)$, where $d_E\in \Ascr$ and $\Psi\in\Omega^1(i\fsu(E))$, and $i\fsu(E)$ is the bundle of hermitian endomorphisms of $E$.  
 Any connection $D$ on $E$ decomposes uniquely as 
$D=d_E+\Psi$,
 where $d_E$ is a unitary connection (with respect to $h$) and $\Psi\in\Omega^1(i\fsu(E))$.  This realizes $\Dscr\subset T^\ast\Ascr$.
Similarly, using the real isomorphism
$\Omega^1(\slfrak(E))=\Omega^{0,1}(\sufrak(E))$,
we can  associate to a Higgs bundle $(\dbar_E, \Phi)$ the pair $(d_E, \Psi)$, where $d_E=\dbar_E+\partial_E^h$ is the Chern connection of $(E,\dbar_E)$, and $\Psi=\Phi+\Phi^{\ast_h}$. This realizes $\Hscr\subset T^\ast\Ascr$.
 
 At the level of tangent vectors, 
\begin{align}  
 T(T^\ast\Ascr)&=\Omega^1(\sufrak(E))\oplus \Omega^1(i\sufrak(E))\notag \\
 &=\Omega^1(\slfrak(E)) \label{eqn:dr} \\
 &=\Omega^{0,1}(\slfrak(E))\oplus \Omega^{1,0}(\slfrak(E)) \label{eqn:cplx}
 \end{align}
 Under \eqref{eqn:dr}, $\mu\in \Omega^1(\slfrak(E))$ maps to 
$ (\alpha, \psi)$, where
\[
 \xymatrix{\alpha=\tfrac{1}{2}(\mu-\mu^\ast)&\text{and}& \psi=\tfrac{1}{2}(\mu+\mu^\ast)~ .}\]
The complex structure on $X$ induces the identification \eqref{eqn:cplx}. For
  $(\beta,\varphi)\in T_{(\dbar_E,\Phi)}\Hscr$ the identification is given by  $\mu\mapsto(\beta,\varphi)$, where 
\[\xymatrix{\beta=\alpha^{0,1}&\text{and}& \varphi=\psi^{1,0}~.}\]

With this understood, we may write down the hyperk\"ahler structure: $J$ is multiplication by $i$ on $T(T^\ast\Ascr)$ using \eqref{eqn:dr}, whereas $I$ is multiplication by $i$ on $T(T^\ast\Ascr)$ using \eqref{eqn:cplx}. The composition $K=IJ$ defines another complex structure.
Explicitly, extend  the Hodge star operator $\star$  on $X$ to $\Omega^1(\slfrak(E))$ as follows: in a local holomorphic coordinate $z$, write
$\mu\in\Omega^1(\fsl(E))$  as $\mu=Adz+Bd\bar z$  and set
\[\bar\star(\mu)=-iA^*d\bar z+iB^*dz~.\]
Then $I$, $J$, and $K$ are given by:
 \begin{equation}\label{eq complex structures on de Rham}
    \begin{array}{lll}
        I(\mu)&=&\bar\star \mu\\J(\mu)&=&i\mu\\ K(\mu)&=&-i\bar\star\mu~.
    \end{array}
\end{equation}
In terms of $(\beta, \varphi)$, we have
 \begin{equation*}\label{eq complex structures on Higgs}
    \begin{array}{lll}
        I(\beta,\varphi)&=&(i\beta,i\varphi)\\J(\beta,\varphi)&=&(i\varphi^*,-i\beta^*)\\ K(\beta,\varphi)&=&(-\varphi^*,\beta^*)~.
    \end{array}
\end{equation*}

\begin{Remark}
The hyperk\"ahler structure discussed above descends to the moduli space $\Mm_{\rH}$. Indeed, the $\NHC$, Theorem \ref{thm:nhc}, may be recast as the statement that 
$$
\Mm_{\rH}\simeq (\Hscr^{\mathrm{ps}}\cap\Dscr)/\Kcal\ ,
$$
and $\Kcal$ preserves all three complex structures.
\end{Remark}

\begin{Remark}
    Following the terminology introduced in \cite{KapustinWitten},  the letter ``A'' will  be attached to Lagrangian submanifolds, and ``B'' to holomorphic submanifolds. Three letters together refer to the three complex structures $I$, $J$, and $K$, in that order. For example, a submanifold of $\Mm_{\rH}$ is of type ($\BAA$) if it is  holomorphic with respect to $I$, and Lagrangian for $\omega_I^{\CBbb}$, whereas it is of type ($\ABA$) if it is holomorphic with respect to $J$ and Lagrangian with respect to $\omega_J^\CBbb$.
    \end{Remark}

\subsection{Fixed points of the $\C^*$-action} \label{sec:vhs}
The Higgs bundle moduli space $\Mm_\rH$ carries a holomorphic $\C^*$-action which scales the Higgs field:
\begin{equation*}
    \label{eq C*action on Higgs} \xi\cdot[(\bar\p_E,\Phi)]=[(\bar\p_E,\xi\Phi)]~.
\end{equation*}

\begin{Definition}
A Higgs bundle $(\dbar_E,\Phi)$ is said to be a \emph{complex variation of Hodge structure} ($\VHS$) if $[(\dbar_E,\Phi)]$ is a
  fixed point  of the $\CBbb^\ast$-action on $\Mm_{\rH}$.  The notation  $\VHS_\alpha$ will mean a connected component of the fixed point set, labeled by some index $\alpha$.
\end{Definition}

By properness of the  Hitchin fibration \eqref{eqn:hitchin-fibration}, the limit $\lim_{\xi\to 0}[(\bar\p_E,\xi\Phi)]$ always exists in $\Mm_\rH$ and is a $\VHS$. 
Note that 
 $[(\dbar_E,\Phi)]\in\Mm_\rH$ is a $\VHS$ if and only if there exists a $\C^*$-family of gauge transformations $g(\xi)$ so that 
\[(\dbar_E,\Phi)\cdot g(\xi)=(\dbar_E,\xi\Phi)~.\] In particular,  $\lim_{\xi\to 0}[(\bar\p_E,\xi\Phi)]=[(\bar\p_E,0)]$ 
if and only if the holomorphic bundle $(E,\bar\p_E)$ is polystable. 
The following characterizes general fixed points (see \cite[Sec.\ 7]{selfduality}). 
\begin{Proposition}\label{Prop: fixedpoints of C* action}
A point $[(\dbar_E,\Phi)]\in\Mm_\rH$ is a fixed point of the $\C^*$-action if and only if there is a splitting $E=E_1\oplus\cdots\oplus E_\ell$ with respect to which:
\begin{equation}
    \label{eq higgs field vhs}
 \xymatrix{\dbar_E=   \mtrx{\dbar_{E_1}\\&\ddots\\&&\dbar_{E_\ell}}& \text{and}&\ 
    \Phi=\mtrx{0\\\Phi_1&0\\&\ddots&\ddots\\&&\Phi_{\ell-1}&0}~,}
    \end{equation}
where $\Phi_j: E_j\to E_{j+1}\otimes K$ is holomorphic. Note that we allow $\ell=1,$ in which case $\Phi=0.$
\end{Proposition}

\begin{Example}  \label{ex:fuchsian}
The \emph{Fuchsian} or \emph{uniformizing} Higgs bundle is defined by setting
$E_j=K^{\frac{n+1}{2}-j}$, and (since $E_{j+1}\otimes K=E_j$) taking $\Phi_j=1$.
\end{Example}

The next result characterizes the flat connections associated to polystable Higgs bundles of the form \eqref{eq higgs field vhs}.
\begin{Proposition}\label{Prop flat conn of fixed point}
    If $(\dbar_E,\Phi)$ is a polystable Higgs bundle of the form \eqref{eq higgs field vhs}, then the splitting  is orthogonal with respect to the harmonic metric $h.$ In particular, the associated flat connection $D=\dbar_E+\partial_E^h+\Phi+\Phi^{\ast_h}$ is given by 
  \[\mtrx{\bar\p_{E_1}&\Phi_1^{\ast_h}\\&\ddots&\ddots\\&&\bar\p_{E_{\ell-1}}&\Phi_{\ell-1}^{\ast_h}\\&&&\bar\p_{E_\ell}}+\mtrx{\p_{E_1}^h\\\Phi_1&\p_{E_2}^h\\&\ddots&\ddots\\&&\Phi_{\ell-1}&\p_{E_l}^h},\]
    where $\Phi_j^{\ast_h}:E_{j+1}\to E_j\otimes \overline K$ is the adjoint of $\Phi_j$ with respect to the induced metrics on $E_j$ and $E_{j+1}$. 
\end{Proposition}

The splitting of $E$ in Proposition \eqref{Prop: fixedpoints of C* action} gives a $\Z$-grading on $\End(E)$, where  
\begin{equation}
    \label{eq End decomp}
    \End_j(E):=\bigoplus\limits_{i-k=j}\Hom(E_i,E_k)~,
\end{equation}  
for $1-\ell\leq j\leq \ell-1$, and we set $\End_j(E)=\{0\}$ otherwise.
We also introduce the following notation:
\begin{equation}
    \label{eq N+ notation}
    \xymatrix@=.5em{N_+=\bigoplus\limits_{j>0} \End_j(E)~,& 
 N_-=\bigoplus\limits_{j<0} \End_j(E) &\text{and}&
 L=\End_0(E)\cap \slfrak(E)~.}
\end{equation}

\subsection{$\lambda$-connections}
Since $\Mm_{\rdR}$ is hyperk\"ahler, its twistor space $\Mm_{\rdR}\times \PBbb^1$ is a complex manifold. The restriction to $\Mm_{\rdR}\times\CBbb$  can equivalently be described as the moduli space of $\lambda$-connections, which we briefly describe here (see \cite{SimpsonHodgeFiltration} for details).
\begin{Definition}\label{Def: lambda connection}
     A \emph{$\lambda$-connection} on $X$ is a triple $(\lambda,\bar\p_E,\nabla_\lambda)$, where  $\lambda\in\C$, $\dbar_E$ is a $\dbar$-operator on $E$, and $\nabla_\lambda:\Omega^0(E)\to \Omega^{1,0}(E)$ is a differential operator satisfying
    \begin{itemize}
         \item a $\lambda$-scaled Leibniz rule:
    \[\nabla_\lambda(fs)=\lambda\cdot \partial f\otimes s+f\cdot \nabla_\lambda s,\] for all smooth functions
    $f$ and $s\in \Omega^0(E),$ and 
    \item  compatibility with the holomorphic structure: $[\bar\partial_E,\nabla_\lambda]=0.$ 
    \end{itemize} 
\end{Definition}
Thus, when $\lambda=1$, a $\lambda$-connection  is just a holomorphic connection $\nabla: \Ecal\to \Ecal\otimes K$ on the holomorphic bundle $\Ecal=(E,\dbar_E)$ defined by $\dbar_E$, which is equivalent to a flat connection  $D=\dbar_E+\nabla$ on $E$. When $\lambda=0$, a $\lambda$-connection is an endomorphism $\Phi : \Ecal\to \Ecal\otimes K$, and the data $(\dbar_E, \Phi)$ gives a Higgs bundle.
Isomorphisms of $\lambda$-connections are defined via complex gauge transformations just as with holomorphic connections.
The moduli space of (polystable) $\lambda$-connections $[(\lambda,\bar\p_E, \nabla_\lambda)]$ is denoted by $\Mm_\Hod$. 
The tautological  map 
$$
\Mm_{\Hod}\lra \CBbb : [(\lambda,\bar\p_E,\nabla_\lambda)]\mapsto \lambda
$$
will (somewhat abusively) be called $\lambda$.
 From the comments above, $\lambda^{-1}(0)=\Mm_{\rH}$ and $\lambda^{-1}(1)=\Mm_{\rdR}$.

There is a $\C^*$-action on $\Mm_{\Hod}$  given by
\begin{equation}\label{eq c^* action on Hod}
    \xi\cdot [(\lambda,\bar\p_E,\nabla_\lambda)]=[(\xi\lambda,\bar\p_E,\xi \nabla_\lambda)]~.
\end{equation}
The only fiber preserved by this action is $\lambda^{-1}(0)$, and the restriction to $\lambda^{-1}(0)$
 is the  $\C^*$-action on Higgs bundles of the previous section. Hence, the  fixed points on $\Mm_{\Hod}$ are exactly the locus of complex variations of Hodge structures. 

In \cite{SimpsonHodgeFiltration}, Simpson shows that the limit as $\xi\to0$ in \eqref{eq c^* action on Hod} defines a 
Bia{\l}ynicki-Birula-type stratification of $\Mm_{\Hod}$.
If $\coprod\limits_{\alpha}\VHS_\alpha$ is the decomposition of $\VHS$ into connected components, then define 
\begin{equation*}
    \label{eq N alpha strata}
    \rW_\alpha=\{[(\lambda,\bar\p_E, \nabla_\lambda)]\in\Mm_\Hod\ |\ \lim\limits_{\xi\to0}[(\xi\lambda,\bar\p_E,\xi \nabla_\lambda)]\in \VHS_\alpha\}~.
\end{equation*}
Each $\rW_\alpha$ is connected and foliated by $\C^*$-orbits, and $\Mm_\Hod=\coprod\limits_{\alpha}\rW_\alpha$. The restrictions $\rW^0_\alpha=\rW_\alpha\cap \lambda^{-1}(0)$ and $\rW^1_\alpha=\rW_\alpha\cap \lambda^{-1}(1)$
 define  stratifications of $\Mm_\rH$ and $\Mm_\rdR$ respectively:
\[\xymatrix{\Mm_\rH=\coprod\limits_\alpha \rW_\alpha^0&\text{and}&\Mm_\rdR=\coprod\limits_\alpha \rW_\alpha^1}~.\]
\begin{Remark}
    The open stratum $\rW_{\mathrm{open}}$ fibers over the component of $\VHS$ associated to the moduli space of polystable bundles. Namely, 
    \[\lim\limits_{\xi\to 0}[(\xi\lambda,\bar\p_E,\xi \nabla_\lambda)]=[(\bar\p_E,0)]\]
    if and only if the holomorphic bundle $(E,\bar\p_E)$ is polystable. In this case, the open stratum $\rW^0_{\mathrm{open}}\subset\Mm_\rH$ is the cotangent sheaf of the moduli space of polystable bundles. 
\end{Remark}
The map $\Tcal$ from \eqref{eq NAH intro} actually
 produces a family of flat connections 
\[\dbar_E+\partial_E^h+\xi^{-1}\Phi+\xi\Phi^{*_h}\]
parametrized by $\xi\in \CBbb^\ast$. 
If we furthermore use the $\C^*$-action \eqref{eq c^* action on Hod} on $\lambda$-connections we obtain
\begin{equation}
    \label{eq twistor line}
    \xi\cdot[(1,\dbar_E+\xi\Phi^{*_h},\partial_E+\xi^{-1}\Phi)]=[(\xi,\dbar_E+\xi\Phi^{*_h},\xi\partial_E+\Phi)]~.
\end{equation}
Moreover, 
\[\lim\limits_{\xi\to0}[(\xi,\dbar_E+\xi\Phi^{*_h},\xi\partial_E^h+\Phi)]=[(\dbar_E,\Phi)]~.\] 

The extension of the $\C^*$-family \eqref{eq twistor line} to $\C$ is an example of a {\em real twistor line} in $\Mm_{\Hod}.$ Thus, the above real twistor line through a Higgs bundle $[(\bar\p_E,\Phi)]$ is a section of $\Mm_\Hod$ which interpolates between the Higgs bundle and the flat connection $\Tt(\bar\p_E,\Phi)$ from the nonabelian Hodge correspondence.

The real twistor line through a Higgs bundle $[(\bar\p_E,\Phi)]$ and the $\C^*$-orbit through $\Tt(\bar\p_E,\Phi)$ coincide if and only if $(\dbar_E,\Phi)$ is a $\VHS$. In particular, 
$$\Tt(\bar\p_E,\Phi)\in \rW^1_\alpha(\bar\p_E,\Phi)\ \text{ if }\ [(\bar\p_E,\Phi)]\in\VHS_\alpha\ .$$
However, the nonabelian Hodge correspondence does not preserve the stratification in general. That is, $\Tt$ does not map $\rW^0_\alpha$ to $\rW^1_\alpha$.
\subsection{The Hitchin section and opers}\label{sec: HitchinSection and Opers}
We now recall the explicit parameterizations of the Hitchin section and the space of opers by the vector space $\bigoplus\limits_{j=1}^{n-1}H^0(K^{j+1})$.

The starting point is the uniformizing rank two $\VHS$ from Example \ref{ex:fuchsian} 
\begin{equation}
    \label{eq fuchsian sl2}\left(K^\haf\oplus K^{-\haf},\mtrx{0&0\\1&0}\right)~.
\end{equation}
Taking the $(n-1)^{\mathrm{st}}$-symmetric product, renormalizing and twisting by an $n^{\mathrm{th}}$-root of the the trivial bundle defines the following $\VHS$
\begin{equation}
    \label{eq VHS oper hitchin section}(E,\bar\p_0,\Phi_0)=\left(L\oplus (L\otimes K^{-1})\oplus\cdots\oplus (L\otimes K^{1-n})~,~\mtrx{0&\\1&0\\&\ddots&\ddots\\&&1&0}\right)~,
\end{equation}
where $L^n\otimes K^{\frac{n(n-1)}{2}}=\Oo.$ Note that the set of such $\VHS$ is finite of cardinality $n^{2g}.$ These are called  \emph{Fuchsian points}, since they are obtained from the Higgs bundle \eqref{eq fuchsian sl2} whose associated flat connection corresponds to the Fuchsian representation that uniformizes the Riemann surface $X$ (cf. \cite{selfduality}). 

By Proposition \ref{Prop flat conn of fixed point}, the flat connection $\Tt(\bar\p_0,\Phi_0)=\dbar_0+\Phi_0^{\ast_h}+\p_0^h+\Phi_0$ is given by 
    \[\mtrx{\bar\p_{L}&1^{*_h}\\&\ddots&\ddots\\&&\bar\p_{L\otimes K^{n-2}}&{1^{*_h}}\\&&&\bar\p_{L\otimes K^{n-1}}}+\mtrx{\p_{L}^h\\1&\p_{L\otimes K^{-1}}^h\\&\ddots&\ddots\\&&1&\p_{L\otimes K^{n-1}}^h},\]

As in \eqref{eq End decomp}, the holomorphic splitting of $E$ gives a decomposition $\bigoplus\limits_{j=1-n}^{n-1}\End_j(E)$ of the endomorphism bundle. For the Higgs bundle \eqref{eq VHS oper hitchin section} we have 
\[\End_j(E)\cong \underbrace{K^{j}\oplus\cdots\oplus K^j}_{n-|j|}~.\]
With respect to this splitting we have 
\[\ad_{\Phi_0^{*_h}}:\End_j(E)\otimes K\to\End_{j+1}(E)\otimes K \bar K~.\] Moreover, $\ker(\ad_{\Phi_0^{*_h}})\simeq K^{j+1}$ for $1\leq j \leq n-1.$
Set 
\begin{equation*}
    \label{eq kernel of raising}V_{j}\otimes K = \ker(\ad_{\Phi_0^{*_h}})\simeq K^{j+1}\hookrightarrow\End_{j}(E)\otimes K~.
\end{equation*} 

With the above notation and a fixed choice of $L,$ consider the maps 
\begin{align}
\begin{split} \label{eqs oper Hitchin param}
p_{\mathrm{Hit}} : \bigoplus\limits_{j=1}^{n-1} H^0(K^{j+1}) &\lra \Mm_{\rH}:  \\
&
(q_2,\cdots,q_n)\mapsto [(\dbar_0, \Phi_0+q_2+\cdots+q_n)] \ ,\\
p_{\mathrm{oper}} : \bigoplus\limits_{j=1}^{n-1} H^0(K^{j+1}) &\lra \Mm_{\rdR} : \\
& (q_2,\cdots,q_n)\mapsto [(\dbar_0+\Phi^{*_h}_0,\p_0^h+ \Phi_0+q_2+\cdots+q_n)]  \ ,
\end{split}
\end{align}
where for $1\leq j\leq n-1$ and $q_{j+1}\in H^0(V_j\otimes K)\subset H^0(\End_j(E)\otimes K)$. 
 
   The following theorem was proven in \cite{liegroupsteichmuller} for $p_\mathrm{Hit}$ and in \cite{BeilinsonDrinfeldOPERS} for $p_{\mathrm{oper}}.$  
\begin{Theorem}
    Let $[(\bar\p_0,\Phi_0)]\in\VHS$ be given by \eqref{eq VHS oper hitchin section} and let $p_\mathrm{Hit}$ and $p_\mathrm{oper}$ be as in \eqref{eqs oper Hitchin param}. 
    \begin{enumerate}
        \item The map $p_\mathrm{Hit}$ is a holomorphic embedding onto $\rW^0(\bar\p_0,\Phi_0)$, which is closed in $\Mm_\rH.$
        \item The map $p_\mathrm{oper}$ is a holomorphic embedding onto $\rW^1(\bar\p_0,\Phi_0)$, which is closed in $\Mm_\rdR.$
    \end{enumerate}
\end{Theorem}
\begin{Remark}
  For each choice of line bundle $L$, the image of $p_\mathrm{Hit}$ is called a \emph{Hitchin component} and the image of the map $p_\mathrm{oper}$ is called a connected component of the space of \emph{opers}. 
\end{Remark}

\section{Deformation theory}
\label{sec: def theory section}

\subsection{The Hodge slice} \label{sec:hodge-slice}
In this section we describe local Kuranishi models for $\Mm_{\rdR}$ and $\Mm_{\rH}$. Let $(\dbar_E, \Phi)$ be a polystable Higgs bundle with harmonic metric $h$ and Chern connection $\dbar_E+\partial_E^h$. Thus, the connection $D=\bar\p_E+\p_E^h+\Phi^{\ast_h}+\Phi$ is flat. Since $h$ will be fixed throughout this section we omit it from the notation.

Set
\begin{equation*} \label{eqn:dprime}
\xymatrix{D=D'+D''\ ,& D'' := \dbar_E+\Phi&\text{and}& D':=\partial_E+\Phi^\ast~.}
\end{equation*}
A key fact is that these operators satisfy the usual K\"ahler identities:
\begin{equation} \label{eqn:kahler}
\xymatrix{(D'')^\ast=-i[\Lambda, D']&\text{and}& (D')^\ast=+i[\Lambda, D'']~ .}
\end{equation}
The \emph{deformation complex} is then given by:
\begin{equation} \label{eqn:complex1}
C(\dbar_E, \Phi) :  \Omega^0(\slfrak(E))\stackrel{D''}{\lra} \Omega^{0,1}(\slfrak(E))\oplus \Omega^{1,0}(\slfrak(E)) \stackrel{D''}{\lra}\Omega^{1,1}(\slfrak(E))\ .
\end{equation}
By \eqref{eqn:kahler}, $\ker (D'')^\ast=\ker D'$, so the harmonic representation of $H^1(C(\dbar_E, \Phi))$ is
\begin{align}
\begin{split} \label{eqn:H1}
\Hcal^1(\dbar_E,\Phi):=\bigl\{ (\beta,\varphi)\in \Omega^{0,1}(\slfrak(E))&\oplus \Omega^{1,0}(\slfrak(E)) \mid \\
& D''(\beta,\varphi)=0\ ,\ 
D'(\beta,\varphi)=0\bigr\}~.
\end{split}
\end{align}
At a stable Higgs bundle $(\dbar_E,\Phi)$, we have $H^i(C(\dbar_E, \Phi))=\{0\}$, $i=0,2$, and 
$$
T_{[(\dbar_E,\Phi)]}\Mm_{\rH}\simeq \Hcal^1(\dbar_E,\Phi)\ .
$$
Note that:
$$
D^\ast=(D')^\ast+(D'')^\ast=-i\Lambda D^c\ ,
$$
where $D^c=D''-D'$. Via the identification $\mu=\beta+\varphi$, we have also
$$
\Hcal^1(\dbar_E,\Phi)=\left\{ \mu \in \Omega^1(\slfrak(E)) \mid D\mu=0\ ,\ D^\ast\mu=0\right\}\ ,
$$
which for irreducible $D$ also gives the identification:
$$
T_{[D]}\Mm_{\rdR}\simeq \Hcal^1(\dbar_E,\Phi)\ .
$$

\begin{Definition} \label{def:slice}
The \emph{Hodge slice} (or simply \emph{slice}) at $(\dbar_E,\Phi)\in \Hscr$ is  given by 
\begin{align*}
\Scal_{(\dbar_E,\Phi)}=\bigl\{
(\beta,\varphi)\in 
\Omega^{0,1}(\slfrak(E))&\oplus \Omega^{1,0}(\slfrak(E)) \mid\\
&\qquad  D''(\beta,\varphi)+[\beta,\varphi]=0\ ,\ 
D'(\beta,\varphi)=0
\bigr\}\ .
\end{align*}

\end{Definition}
Notice that  $\Scal_{(\dbar_E,\Phi)}\subset \Hscr$ under the identification $(\beta,\varphi)\mapsto (\dbar_E+\beta, \Phi+\varphi)$.
 Moreover, at points where $\Hscr$ is smooth and $\Scal_{(\dbar_E,\Phi)}$ is a submanifold, then $\Scal_{(\dbar_E,\Phi)}$ is a holomorphic submanifold. By direct computation we have the following:
\begin{Lemma} \label{lem:slice}
The map $(\beta, \varphi)\mapsto D+\beta+\varphi$ defines a locally closed, holomorphic embedding $\Scal_{(\dbar_E,\Phi)}\hookrightarrow \Dscr$.
\end{Lemma}
We shall often view $\Scal_{(\dbar_E,\Phi)}$ as a subset of $\Dscr$ via this embedding.

\begin{Remark}
The subset $\Scal_{(\dbar_E,\Phi)}\subset\Dscr$ does not coincide with the usual \emph{de Rham slice} to $\Dscr$ at $D$. The latter is defined by requiring $D+\beta+\varphi$ to be flat and $D^\ast(\beta+\varphi)=0$. The first of these equations is satisfied by points in $\Scal_{(\dbar_E,\Phi)}$. The second equation, however,  is equivalent to $D^c(\beta,\varphi)=0$, whereas the second equation defining $\Scal_{(\dbar_E,\Phi)}$ is $D'(\beta,\varphi)=0$. Of course, the tangent spaces   at the origin of the two slices agree.
\end{Remark}

Let
\begin{align}
\begin{split} \label{eqn:p}
p_{\rH} : \Scal_{(\dbar_E,\Phi)}\cap \Hscr^{\rm ps} &\lra \Mm_{\rH} : (\beta,\varphi)\mapsto [(\dbar_E+\beta, \Phi+\varphi)] \ ,\\
p_{\rdR} : \Scal_{(\dbar_E,\Phi)}\cap \Dscr^{\rm cr} &\lra \Mm_{\rdR} : (\beta,\varphi)\mapsto [D+\beta+\varphi] \ .
\end{split}
\end{align}
be the projection maps to moduli.  Note that if $(\dbar_E,\Phi)$ is stable (or equivalently, $D$ is irreducible), then the maps $p_{\rH}$ and $p_{\rdR}$ are well-defined in a neighborhood of the origin in $\Scal_{(\dbar_E,\Phi)}$.
For completeness, we state without proof the following result which shows that the slice defines a local coordinate chart for the moduli space (cf.\ \cite[Thm.\ 7.3.17]{DiffGeomCompVectBun}). We shall prove a stronger statement for a restricted slice in the next section.
\begin{Proposition} \label{prop:slice1}
If $(\dbar_E,\Phi)$ is stable, then $p_{\rH}$ $($resp.\ $p_{\rdR}$$)$ is a homeomorphism from a neighborhood of the origin in $\Scal_{(\dbar_E,\Phi)}$ to a neighborhood of $[(\dbar_E,\Phi)]\in\Mm_{\rH}$ $($resp.\ $[D]\in\Mm_{\rdR}$$)$.
\end{Proposition}

The \emph{Kuranishi map} is defined as follows:
\begin{equation} \label{eqn:kuranishi}
k(\beta, \varphi) := (\beta,\varphi)+ (D'')^\ast G([\beta,\varphi])\ ,
\end{equation}
where $G$ is the Green's operator for the Laplacian $D''(D'')^\ast=iD''D'\Lambda$ (see \eqref{eqn:kahler}) acting on $\Omega^{1,1}(\slfrak(E))$. Notice that if $(\beta,\varphi)\in \Scal_{(\dbar_E,\Phi)}$, then $[\beta,\varphi]\perp\Hcal^2(C(\dbar,\Phi))$ and  $k(\beta,\varphi)\in \Hcal^1(C(\dbar,\Phi))$.  By the implicit function theorem we have (cf.\ \cite[Thm.\ 7.3.23]{DiffGeomCompVectBun}):
\begin{Proposition} \label{prop:kuranishi1}
If $(\dbar_E,\Phi)$ is stable, then $k: \Scal_{(\dbar_E,\Phi)}\to \Hcal^1(C(\dbar,\Phi))$ is a homeomorphism in a neighborhood of the origin.
\end{Proposition}

\subsection{The $\BB$-slice at a $\VHS$} \label{sec:bb-slice}
Suppose $(\dbar_E,\Phi)$ is a $\VHS$ as in Section \ref{sec:vhs}, and recall the notation of \eqref{eq End decomp} and \eqref{eq N+ notation}.
Since $\Phi\in H^0(\End_{-1}(E)\otimes K)$, we have
\[\ad_\Phi:\End_j(E)\to\End_{j-1}(E)\otimes K~.\]
Since $\dbar_E$ preserves the grading, for each $j$ we get a subcomplex $C_j(\dbar_E, \Phi)$ of \eqref{eqn:complex1} given by:
$$
\Omega^0(\End_{j}(E))\stackrel{D''}{\lra} \Omega^{0,1}(\End_{j}(E))\oplus \Omega^{1,0}(\End_{j-1}(E)) \stackrel{D''}{\lra}\Omega^{1,1}(\End_{j-1}(E))\ .
$$
Let $\Hcal^1_j(\dbar_E,\Phi):=\ker D''\cap\ker D'$ for the middle term in $C_j(\dbar_E, \Phi)$.
This gives a grading $\Hcal^1(\dbar_E,\Phi)=\bigoplus\limits_{1-\ell}^{\ell}\Hcal^1_j(\dbar_E,\Phi)$.
We will use the following notation:
\begin{equation}
    \label{eq HH1+ notation}
   \Hcal^1_+(\dbar_E,\Phi))=\bigoplus\limits_{j=1}^{\ell} \Hcal^1_j(\dbar_E,\Phi))~.
\end{equation}
In particular, $\Hcal^i_+(\dbar_E,\Phi)\simeq H^i(C_+(\dbar_E, \Phi))$ for the subcomplex
$C_+(\dbar_E, \Phi)$
of \eqref{eqn:complex1} given by
\begin{equation} \label{eqn:complex2}
  \Omega^0(N_+)\stackrel{D''}{\lra} \Omega^{0,1}(N_+)\oplus \Omega^{1,0}(L\oplus N_+) \stackrel{D''}{\lra}\Omega^{1,1}(L\oplus N_+)\ .
\end{equation}

\begin{Lemma}  \label{lem:dimension}
If $(\dbar_E,\Phi)$ is a stable $\VHS$, then $\Hcal^1_+(\dbar_E,\Phi)$ is half-dimensional.
\end{Lemma}

\begin{proof}
Roll up the complex \eqref{eqn:complex2} to obtain an operator 
$$\slash \hskip-.095in D := (D'')^\ast+D'' : \Omega^{0,1}(N_+)\oplus \Omega^{1,0}(L\oplus N_+) 
\lra  \Omega^0(N_+)\oplus \Omega^{1,1}(L\oplus N_+)\ .
$$
Since $(\dbar_E,\Phi)$ is  stable, $H^0(C_+(\dbar_E, \Phi))=H^2(C_+(\dbar_E, \Phi))=\{0\}$, and so 
$$
\dim H^1(C_+(\dbar_E, \Phi))=\ind\, \slash \hskip-.096in D\ .
$$
Deforming the Higgs field to zero does not change the index but it does decouple the operators. We therefore have
$$\ind\, \slash \hskip-.096in D = \ind\, \dbar_0 - \ind\, \dbar_1\ ,
$$
where $\dbar_0$ is the $\dbar$-operator on sections of $(L\oplus N_+)\otimes K$ induced by $\dbar_E$, and $\dbar_1$ is similarly  the induced $\dbar$-operator on sections of $N_+$. Since $\deg L=0$, we have by Riemann-Roch:
\begin{align*}
 \ind\, \dbar_0 &= \deg N_+ +(\rank L +\rank N_+)(g-1) \\
  \ind\, \dbar_1 &= \deg N_+ -(\rank N_+)(g-1) \ .
\end{align*}
Notice that 
$$n^2-1=\rank(\fsl( E)) = \rank L + 2\rank N_+\ ,$$
so that 
$$
\ind\, \slash \hskip-.096in D = (n^2-1)(g-1)=\tfrac{1}{2}\dim \Mm_{\rH}\ .
$$
\end{proof}

\begin{Definition} \label{def:BBslice}
The \emph{$\BB$-slice} at $(\dbar_E,\Phi)$ is defined as
\begin{align*}
\Scal^+_{(\dbar_E,\Phi)}&=\bigl\{
(\beta,\varphi)\in 
\Omega^{0,1}(N_+)\oplus \Omega^{1,0}(L\oplus N_+) \mid \\
&\qquad\qquad\qquad D''(\beta,\varphi)+[\beta,\varphi]=0\ ,\ 
D'(\beta,\varphi)=0
\bigr\}\ .
\end{align*}
\end{Definition}

If $(\dbar_E, \Phi)$ is stable, then $\Scal^+_{(\dbar_E,\Phi)}\subset \Scal_{(\dbar_E,\Phi)}$ is a half-dimensional complex submanifold in a neighborhood of the origin.  
\begin{Example}
    \label{ex:fuchsian-slice} Consider a Fuchsian point $(\bar\p_0,\Phi_0)$ from \eqref{eq VHS oper hitchin section} and recall the maps $p_\mathrm{Hit}$ and $p_\mathrm{oper}$ from \eqref{eqs oper Hitchin param}. For $(q_2,\cdots,q_n)\in\bigoplus\limits_{j=1}^{n-1}H^0(K^{j+1})$, we have
    \[p_\mathrm{Hit}(q_2,\cdots,q_n)-(\bar\p_0,\Phi_0)=(0~,~q_2+\cdots+q_n)\in \Omega^{0,1}(N_+)\oplus \Omega^{1,0}(L\oplus N_+)~,\]
    \[p_\mathrm{oper}(q_2,\cdots,q_n)-(\bar\p_0+\Phi_0^*,\p_0^h+\Phi_0)=(0~,~q_2+\cdots+q_n)\in \Omega^{0,1}(N_+)\oplus \Omega^{1,0}(L\oplus N_+)~.\]    

    Moreover, both $p_\mathrm{Hit}-(\bar\p_0,\Phi_0)$ and $p_\mathrm{oper}-(\bar\p_0+\Phi_0^*,\p_0^h+\Phi_0)$ map $\bigoplus\limits_{j=1}^{n-1}H^0(K^{j+1})$ bijectively onto both $\Hh^1_+(\bar\p_0,\Phi_0)$ and $\Ss_{(\bar\p_0,\Phi_0)}^+.$ 
    Indeed, $\bigoplus\limits_{j=1}^{n-1}H^0(K^{j+1})$ and
    $\Hh^1_+(\bar\p_0,\Phi_0)$
      have the same dimension and, by definition of the maps $p_\mathrm{Hit}$ and $p_\mathrm{oper}$, the image of each $q_j$ is in the kernel of $\ad_{\Phi_0^{*_h}}$. Hence,  we have 
    \[p_\mathrm{Hit}(q_2,\cdots,q_n)-(\bar\p_0,\Phi_0)\in \ker(D')\cap\ker(D'')~,\]
    \[p_\mathrm{oper}(q_2,\cdots,q_n)-(\bar\p_0+\Phi_0^*,\p_0+\Phi_0)\in \ker(D')\cap\ker(D'')~.\]  
and since $\beta=0$, these points are automatically in $\Ss^+_{(\bar\p_0,\Phi_0)}$ as well.
Thus, for a Fuchsian point, the restrictions of the maps $p_\rH$ and $p_\rdR$ to $\Ss^+_{(\bar\p_0,\Phi_0)}$ recover the parameterization of the Hitchin sections and the components of opers by the affine space $\bigoplus\limits_{j=1}^{n-1}H^0(K^{j+1})$. 
\end{Example}
While $\Scal_{(\dbar_E,\Phi)}$ gives only a local chart in $\Mm_{\rH}$, we shall see that $\Scal^+_{(\dbar_E,\Phi)}$ extends globally. First, we prove

\begin{Theorem} \label{thm:kuranishi}
The Kuranishi map \eqref{eqn:kuranishi} gives a biholomorphism 
$$k: \Scal^+_{(\dbar_E,\Phi)}\isorightarrow \Hcal^1_+(\dbar_E,\Phi)\ .$$
\end{Theorem}

The proof requires the following
\begin{Lemma}
    The $\C^*$-action on $\Omega^{0,1}(\End(E))\oplus\Omega^{1,0}(\End(E))$ defined by
\begin{equation}
    \label{eq C* action on slice}\xi\cdot\Bigl(\sum\limits_{j=1-\ell}^{\ell-1}\beta_j~,~\sum\limits_{j=1-\ell}^{\ell-1}\varphi_j\Bigr)=\Bigl(\sum\limits_{j=1-\ell}^{\ell-1}\xi^j\beta_j~,~\sum\limits_{j=1-\ell}^{\ell-1}\xi^{j+1}\varphi_j\Bigr)~
\end{equation}
preserves both the $\BB$-slice $\Scal^+_{(\dbar_E,\Phi)}$ and $\Hcal^1_+(\dbar_E,\Phi)$.
\end{Lemma}
\begin{proof}
 The $j^{th}$-graded piece of $D''(\beta,\varphi)+[\beta,\varphi]$ is given by 
\[\bar\p_E\varphi_j+[\beta_{j+1},\Phi]+\sum\limits_{a+b=j}[\beta_a,\varphi_b]\]
and the $\C^*$-action \eqref{eq C* action on slice} scales this equation by $\xi^{j+1}.$ 
Similarly, the $j^{th}$-graded piece of $D'(\beta,\varphi)$ is given by 
$\p_E\beta_j+[\Phi^*,\varphi_{j-1}]$,
and the $\C^*$-action scales this by $\xi^j.$
\end{proof}

\begin{proof}[Proof of Theorem \ref{thm:kuranishi}]
The map \eqref{eqn:kuranishi} is clearly holomorphic. 
By Proposition \ref{prop:kuranishi1}, 
it is a  biholomorphism in a neighborhood of the origin. 
The $\CBbb^\ast$-action defined in \eqref{eq C* action on slice} preserves the slice and $\Hcal^1_+(\dbar_E,\Phi)$, and the Kuranishi map is equivariant. It follows that the local biholomorphism extends to a global one.
\end{proof}

We now come to an important technical result.
\begin{Proposition} \label{prop:good-gauge}
Assume $(\dbar_E,\Phi)$ is stable.
For each $$(\beta,\varphi)\in 
\Omega^{0,1}(N_+)\oplus \Omega^{1,0}(L\oplus N_+)$$ satisfing
$D''(\beta,\varphi)+[\beta,\varphi]=0$, there is a unique $f\in \Omega^0(N_+)$ such that the complex gauge transformation $g=\Ibold+f$ takes $(\bar\p_E+\beta,\Phi+\varphi)$ into the slice $\Scal^+_{(\dbar_E,\Phi)}$.
\end{Proposition}

\begin{proof}
Write $f$ in terms of its graded pieces, $f=f_1+\cdots + f_{\ell-1}$, where
$f_j\in \End_j (E)$.
Then the grading is additive under multiplication. 
By the phrase ``mod $j$'' we shall mean any expression involving only sections $f_k$ for $k<j$. 
For example, if we write: $g^{-1}=\Ibold+h_1+\cdots + h_{\ell-1}$ in terms of graded pieces, then $h_j=-f_j \mod j$. 
In the following,
suppose $R(f)$  is an algebraic expression in $f$ and its derivatives $\dbar_E f$, and such that  the $j$-th graded piece $R(f)_j$ depends only on $f_k$, $k\leq j$. We will say that $R(f)_j=S(f_j)\mod j$, where $S$ is another such algebraic expression,   if $R(f)_j-S(f_j)$ depends only on $f_k$, $k\leq j-1$. 

The strategy is to solve recursively for the graded pieces of $f_j$ of $f$, starting with $j=1$. 
With the conventions of the previous paragraph understood, 
the effect of the gauge transformation $g=\Ibold+f$ on the graded pieces of the data is:
\begin{align*}
\beta_j&\mapsto \beta_j+\dbar_E (f_j) \mod j \\
[\Phi^\ast, \varphi]_j&\mapsto [\Phi^\ast, \varphi]_j+\left([\Phi^\ast, [\Phi, f]]\right)_j\mod j \ .
\end{align*}
Hence, we wish to find $f$ with components $f_j$ solving the equation:
$$
\partial_E\dbar_E (f_j)+\left([\Phi^\ast, [\Phi, f]]\right)_j=-(\partial_E \beta+[\Phi^\ast, \varphi])_j \mod j
$$
But the left hand side is just $D'D''(f_j)\mod j$.
Contracting with $-\sqrt{-1}\Lambda$, and using \eqref{eqn:kahler},  this becomes
\begin{equation} \label{eqn:f}
(D'')^\ast D'' (f_j) =-2i\Lambda(\partial_E \beta+[\Phi^\ast, \varphi])_j \mod j~.
\end{equation}
By the assumption of stability, $\ker D''=\{0\}$ on $\Omega^0(\slfrak(E))$, and so  the operator $(D'')^\ast D''$
 is invertible. Now assuming we have solved for $f_k$, $k<j$,  \eqref{eqn:f} gives a unique solution for $f_j$. This completes the proof. 
 \end{proof}

\subsection{First variation of the harmonic metric}
Let $(\bar\p_E,\Phi)$ be a stable Higgs bundle. For this section, we denote  the harmonic metric
on $(\bar\p_E,\Phi)$ by $h_0$. For each $(\beta,\varphi)\in \Scal_{(\dbar_E,\Phi)}$, the Higgs bundle 
$(\dbar_E+\beta, \Phi+\varphi)$ also admits a harmonic metric in a neighborhood of the origin. We view this as a function $h$ from  $\Scal_{(\dbar_E,\Phi)}$ to the space $\Met(E)$ of hermitian metrics on $E$.  A standard fact is that $h$  is smooth in a 
neighborhood of the origin. The following is a generalization of  \cite[Thm.\ 3.5.1]{LabourieWentworthFuchLocus}. 
\begin{Proposition}\label{Prop first variation of metric}
  The derivative of the harmonic metric $h:\Scal_{(\dbar_E,\Phi)}\to \Met(E)$ vanishes at the origin.
\end{Proposition}

\begin{proof}
We will consider one parameter variations (in the variable $t$) in the direction $(\beta,\varphi)\in \Hcal^1(\dbar_E,\Phi)$ and use notation like $\dt \ubold$ to denote derivatives at $t=0$.
Let
$\ubold(t)\in \Scal_{(\dbar_E,\Phi)}$ be a one parameter family such that $\ubold(0)=0$ and $\dt \ubold=(\beta,\varphi)$.
Let $h(t)$ denote the harmonic metric for the Higgs bundle defined by $\ubold(t)$.
Write $h(t)=h_0H(t)$ for an invertible hermitian endomorphism valued function $H$, with $\det H(t)=1$ and  $H(0)=\Ibold$.
It suffices to show that $\dt H=0$, $H=H(t)$,  for one parameter  variations  in the direction $(\beta,\varphi)\in \Hcal^1(\dbar_E,\Phi)$. In the following, $\ast$ will denote adjoints with respect to $h_0$, whereas $\ast_h$ will denote the same for $h$. 
 Furthermore, $F_{(\dbar_E,h)}$ denotes the curvature of the Chern connection of $(E,\dbar_E)$ with respect to $h$, and 
$\partial_E:=\partial_E^{h_0}$.

We first claim that
\begin{equation} \label{eqn:partial-dot}
\stackrel{\sbt}{\widefrown{\partial_E^h}}\, =\partial_E(\dt H)-\beta^\ast\ .
\end{equation}
Indeed, let $\{e_i\}$ be a local $h_0$-unitary frame, and define
$$
h_{ij}=\langle e_i, e_j\rangle_{h}=\langle He_i, e_j\rangle_{h_0}\ .
$$
Then
$
\dt h_{ij}=\langle \dt He_i, e_j\rangle_{h_0}
$,
so that at $t=0$,
$$
\partial( \dt h_{ij})=\langle \partial_E(\dt He_i), e_j\rangle_{h_0}+ \langle \dt H e_i, \dbar_E e_j\rangle_{h_0}\ .
$$
On the other hand,  
\begin{align*}
\partial h_{ij}&= \langle \partial_E^{h} e_i, e_j\rangle_h+\langle e_i, \dbar_E e_j\rangle_h\\
&= \langle H \partial_E^h e_i, e_j\rangle_{h_0}+\langle H e_i, \dbar_E e_j\rangle_{h_0}\\
\partial(\dt h_{ij})&= \langle \dt H \partial_E e_i, e_j\rangle_{h_0}+ \langle H \stackrel{\sbt}{\widefrown{\partial_E^h}} e_i, e_j\rangle_{h_0}+ \langle \dt H e_i, \dbar_E e_j\rangle_{h_0}+\langle H e_i, \beta e_j\rangle_{h_0}~. \\
\end{align*}
Evaluating at $t=0$ and equating the two results,  we have
$$
\langle \partial_E(\dt H)e_i, e_j\rangle_{h_0}=\langle  \stackrel{\sbt}{\widefrown{\partial_E^h}} e_i, e_j\rangle_{h_0}+\langle H e_i, \beta e_j\rangle_{h_0}\ .
$$
This proves the claim. 

From $\Phi^{\ast_h}=H^{-1}\Phi^{\ast}H$ we also have (at $t=0$)
\begin{equation} \label{eqn:phi-dot}
   \stackrel{\sbt}{\widefrown{\Phi^{\ast_h}}       }\, =(\dt\Phi)^\ast +[\Phi^\ast, \dt H]\ .
\end{equation}
Combining \eqref{eqn:partial-dot} and \eqref{eqn:phi-dot}, 
\begin{align}
\begin{split} \label{eqn:dots}
 \stackrel{\sbt}{\widefrown{F_{(\dbar_E,h)}}}   &=(\dbar_E+\partial_E)(\stackrel{\sbt}{\widefrown{\dbar_E}}+ \stackrel{\sbt}{\widefrown{\partial_E^h}})=\partial_E\beta-\dbar_E\beta^\ast+\dbar_E\partial_E\dt H\ ,\\
\stackrel{\sbt}{\widefrown{[\Phi, \Phi^{\ast_h}]}}&=[\varphi, \Phi^\ast]+[\Phi, \varphi^\ast]+[\Phi, [\Phi^\ast, \dt H]]\ .
\end{split}
\end{align}

Since $h$ is the harmonic metric for $(\dbar_E,\Phi)$, we differentiate
\eqref{eqn:hitchin}  at $t=0$. Plugging into \eqref{eqn:dots}, we obtain
$$
0=D'(\beta,\varphi)-\{D'(\beta,\varphi)\}^\ast+D'' D'(\dt H)\ .
$$
Since $(\beta,\varphi)\in \Hcal^1(\dbar_E,\Phi)$, the first two terms vanish, and so we have 
$$
D''D'(\dt H)=0\ \Longrightarrow\ D'(\dt H)=0\ \Longrightarrow\ D''(\dt H)=0\ ,
$$
since $\dt H$ is hermitian. But  $(\dbar_E,\Phi)$ is stable and $\dt H$ is traceless, so
 this implies $\dt H=0$.
\end{proof}

\section{Stratifications} \label{sec:strata}

\subsection{The stable manifold of a VHS in $\Mm_{\rH}$}
For a fixed point $[(\bar\p_0,\Phi_0)]\in \VHS_\alpha\subset \Mm_\rH$, the \emph{stable manifold} is defined by
\[\rW^0_\alpha(\bar\p_0,\Phi_0)=\left\{[(\bar\p_E,\Phi)]\in\Mm_\rH~|~\lim\limits_{\xi\to0}[(\bar\p_E,\xi\Phi)]=[(\bar\p_0,\Phi_0)]\right\}~.\]
Since stability is an open condition, $\rW^0_\alpha(\bar\p_0,\Phi_0)$ is  a smooth submanifold in the nonsingular locus of $\Mm_{\rH}$ when $(\bar\p_0,\Phi_0)$ is a stable Higgs bundle. Moreover, 
there is a relationship between the grading on the tangent space at a fixed point and its stable manifold via Morse theory. 
Namely, consider the function defined by the $\rL^2$-norm of the Higgs field 
\[f:\xymatrix@R=0em{\Mm_{\rH}\ar[r]&\R^{\geq 0}\\[(\bar\p_E,\Phi)]\ar@{|->}[r]&||\Phi||^2=i\int\limits_X\Tr(\Phi\wedge\Phi^*)}~,\]
where the adjoint is defined with respect to the harmonic metric solving the self duality equations. Using Uhlenbeck compactness, Hitchin showed that $f$ is a proper function. Moreover, on the smooth locus, $f$ is the moment map for the restriction of the $\C^*$-action to $\sU(1)$, and hence a Morse-Bott function \cite{selfduality}. For smooth points, the Hessian of $f$ has eigenvalue $j$ on the subspace $\Hcal^1_j(\dbar_0,\Phi_0)$ (see \cite[Section 8]{liegroupsteichmuller}\footnote{In \cite{liegroupsteichmuller}, Hitchin proves that the Hessian of $f$ has eigenvalue $j$ on the subspace $\Hcal^1_j(\dbar_0,\Phi_0)$. We are using the opposite grading as Hitchin, so the Hessian of $f$ has eigenvalue $j$ on the subspace $\Hcal^1_j(\dbar_0,\Phi_0)$.}). 
Using the notation of \eqref{eq HH1+ notation}, this implies that 
\[T_{[(\bar\p_0,\Phi_0)]}\rW^0_\alpha(\bar\p_0,\Phi_0)\cong\Hcal^1_{+}(\dbar_0,\Phi_0)~,\]
and by Lemma \ref{lem:dimension} the latter is half dimensional. 
We record these results in the following
\begin{Lemma}\label{Lemma dim of stable man}
    If $[(\bar\p_0,\Phi_0)]\in\VHS_\alpha$ is a stable fixed point of the $\C^*$-action, then $\rW^0_\alpha(\bar\p_0,\Phi_0)$ is a  smooth submanifold of dimension $\haf\dim(\Mm_\rH).$
\end{Lemma}

\begin{Proposition}\label{prop stable man of VHS normal form}
Let $[(\bar\p_0,\Phi_0)]\in\VHS_\alpha$ be a stable variation of Hodge structure, and $[(\bar\p_E,\Phi)]\in\Mm_{\rH}$. Then $[(\bar\p_E,\Phi)]\in\rW^0_\alpha(\bar\p_0,\Phi_0)$ if and only if after a gauge transformation
\[\xymatrix{\Phi-\Phi_0\in \Omega^{1,0}(L\oplus N^+)&\text{and}&\bar\p_E-\bar\p_0\in \Omega^{0,1}(N^+)},\]
where $N_+$ and $L$ are defined in \eqref{eq N+ notation}.
\end{Proposition}
\begin{proof}
    First suppose $\varphi=\Phi-\Phi_0\in \Omega^{1,0}(L\oplus N^+)$ and $\beta=\bar\p_E-\bar\p_0\in \Omega^{0,1}(N^+)$, with respect to the grading we write 
    \begin{equation*}
        \label{eq decomp alpha phi}\xymatrix{\varphi=\sum\limits_{j=0}^{\ell-1}\varphi_j &\text{and}&\beta=\sum\limits_{j=1}^{\ell-1}\beta_j},
    \end{equation*}
    where $\varphi_j\in\Omega^{1,0}(\End_j(E))$ and $\beta_j\in\Omega^{0,1}(\End_j(E))$. Consider the gauge transformations 
\begin{equation}
    \label{eq fixed point hol gauge trans}g_\xi=\mtrx{\xi^a\Ibold_{E_1}&\\&\xi^{a-1}\Ibold_{E_2}\\&&\ddots\\&&&\xi^{a-\ell+1}\Ibold_{E_\ell}}~,
\end{equation}
where $\sum\limits_{j=0}^{\ell-1}\rk(E_{j+1})(a-j)=0.$
The gauge transformation $g_\xi$ has $\det(g_\xi)=1$ and acts on $(\bar\p_E,\xi\Phi)$ as
\[(\bar\p_E,\xi\Phi)\cdot g_\xi=(\bar\p_0+\sum\limits_{j=1}^{\ell-1}\xi^{j}\beta_j~,~\Phi_0+\sum\limits_{j=0}^{\ell-1}\xi^{j+1}\varphi_j)~.
\]
Thus, we have $\lim\limits_{\xi\to 0}[(\bar\p_E,\xi\Phi)]=[(\bar\p_0,\Phi_0)].$

For the other direction, recall that the map $p_{\rH}$ from \eqref{eqn:p} is a diffeomorphism $\Scal_{(\dbar_0,\Phi_0)}$ 
onto an open neighborhood $\Uu\subset\Mm_\rH$ of $[(\bar\p_0,\Phi_0)].$
Using the smooth splitting $E_1\oplus\cdots\oplus E_\ell$ we can decompose $\beta\in\Omega^{0,1}(\End(E))$ and $\varphi\in\Omega^{1,0}(\End(E))$ as $\beta=\sum\limits_{j=1-\ell}^{\ell-1}\beta_j$ and $\varphi=\sum\limits_{j=1-\ell}^{\ell-1}\varphi_j$. 
The $\BB$-slice $\Scal^+_{(\dbar_0,\Phi_0)}\subset \Scal_{(\dbar_0,\Phi_0)}$ is (see Definition \ref{def:BBslice}):
\[\Scal_{(\dbar_0,\Phi_0)}^+=\biggl\{(\beta,\varphi)\in\Scal_{(\dbar_0,\Phi_0)} \mid \beta=\sum\limits^{\ell-1}_{j=1}\beta_j~\text{and}~\varphi=\sum\limits_{j=0}^{\ell-1}\varphi_j\biggr\}~.\]
 By the above argument the portion of $\rW^0_\alpha(\bar\p_0,\Phi_0)$ in $p_{\rH}(\Scal_{(\dbar_0,\Phi_0)})$ is contained in $p_{\rH}(\Scal_{(\dbar_0,\Phi_0)}^+)$:
\[\rW^0_\alpha(\bar\p_0,\Phi_0)\cap p_{\rH}(\Scal_{(\dbar_0,\Phi_0)})\subset p_{\rH}(\Scal_{(\dbar_0,\Phi_0)}^+).\]
In fact, the above inclusion must be an equality, since by Lemmas \ref{lem:dimension} and \ref{Lemma dim of stable man},  $\dim(\Ss_+)=\dim(\rW^0_\alpha(\bar\p_0,\Phi_0))$. 

We now extend this to all of $\rW^0_\alpha(\bar\p_0,\Phi_0).$ For each point $[(\bar\p_E,\Phi)]\in \rW^0_\alpha(\bar\p_0,\Phi_0)$ there exists $\xi\in\C^*$ so that $[(\bar\p_E,\xi\Phi)]\in \Uu=p_{\rH}(\Scal_{(\dbar_0,\Phi_0)}).$ In fact, $[(\bar\p_E,\xi\Phi)]\in p_{\rH}(\Scal_{(\dbar_0,\Phi_0)}^+),$ so there is a representative $(\bar\p_E,\xi\Phi)$ of $[(\bar\p_E,\xi\Phi)]$ such that
\[(\bar\p_E,\xi\Phi)=(\bar\p_0+\sum\limits_{j=1}^\ell\beta_j~,~\Phi_0+\sum\limits_{j=0}^\ell\varphi_j)~.\]
A representative of $[(\bar\p_E,\Phi)]$ is given by 
\[(\bar\p_E,\Phi)=(\bar\p_0+\sum\limits_{j=1}^\ell\beta_j~,~\xi^{-1}\Phi_0+\sum\limits_{j=0}^\ell\xi^{-1}\varphi_j)~.\]
Using the inverse of the holomorphic gauge transformation $g_\xi$ from \eqref{eq fixed point hol gauge trans} we have 
\[(\bar\p_0+\sum\limits_{j=1}^\ell\beta_j~,~\xi^{-1}\Phi_0+\sum\limits_{j=0}^\ell\xi^{-1}\varphi_j)\cdot g_\xi^{-1}=(\bar\p_0+\sum\limits_{j=1}^\ell\xi^{-j}\beta_j~,~\Phi_0+\sum\limits_{j=0}^\ell\xi^{-j-1}\varphi_j)~.\]
Thus, every $[(\bar\p_E,\Phi)]\in \rW^0_\alpha(\bar\p_0,\Phi_0)$ is gauge equivalent to a Higgs bundle $(\bar\p_E,\Phi)$ with $\bar\p_E-\bar\p_0\in\Omega^{0,1}(N_+)$ and $\Phi-\Phi_0\in\Omega^{1,0}(L\oplus N_+).$
\end{proof}

\begin{Corollary} \label{cor:ph}
Suppose $[(\dbar_E,\Phi)]\in \VHS_\alpha$ is stable. Then
the map $p_{\rH}$ in \eqref{eqn:p} gives a biholomorphism $\Scal^+_{(\dbar_E,\Phi)}\isorightarrow \rW^0_\alpha(\dbar_E,\Phi)$.
\end{Corollary}

\begin{proof}
We claim that $p_{\rH}$ is injective. Indeed, this is true in a neighborhood of the origin by Proposition \ref{prop:slice1}. 
In general,
suppose
 $(\beta_1,\varphi_1), (\beta_2,\varphi_2)\in \Scal^+_{(\dbar_E,\Phi)}$ are such that 
$$
(\dbar_E+\beta_1, \Phi+\varphi_1)\cdot g=(\dbar_E+\beta_2, \Phi+\varphi_2)
$$
for some $g\in\Gcal(E)$. Let $\xi\in \CBbb^\ast$ and consider $g_\xi$ from \eqref{eq fixed point hol gauge trans}. Then
as in the proof of Proposition \ref{prop stable man of VHS normal form} we have
$$
(\bar\p_E+\sum\limits_{j=1}^\ell\xi^{j}\beta_{1,j}~,~\Phi+\sum\limits_{j=0}^\ell\xi^{j+1}\varphi_{1,j})\cdot (g_\xi^{-1}\circ g\circ g_\xi)=(\bar\p_E+\sum\limits_{j=1}^\ell\xi^{j}\beta_{2,j}~,~\Phi+\sum\limits_{j=0}^\ell\xi^{j+1}\varphi_{2,j})\ .
$$
For $\xi$ sufficiently small, this is a gauge equivalence between points in a neighborhood of the origin in $\Scal^+_{(\dbar_E,\Phi)}$. Since $p_{\rH}$ is injective there, we conclude that
$(\beta_1,\varphi_1)= (\beta_2,\varphi_2)$, and this proves injectivity of $p_H$ on all of $\Scal^+_{(\dbar_E,\Phi)}$.  Surjectivity of $p_{\rH}$ follows from Propositions \ref{prop stable man of VHS normal form} and \ref{prop:good-gauge}. Since $p_{\rH}$ is in addition holomorphic, the result follows.
\end{proof}
\begin{Remark}\label{rmk: generalization of Hitchin param}
    Note that Corollary \ref{cor:ph} generalizes the parameterization of a Hitchin section from Example \ref{ex:fuchsian-slice} to the fibers of every smooth strata.
\end{Remark}

\subsection{The partial oper stratification of $\Mm_{\rdR}$}
Let $\Ecal$ be a bundle with holomorphic connection $\nabla$. We write
 $(\Ee, \nabla)=(E,\bar\p_E,\nabla)$.
Recall that $\Ecal$ is polystable if and only if the $\C^*$-action from \eqref{eq c^* action on Hod} has limit 
\[\lim\limits_{\xi\to0}[(\xi,\bar\p_E,\xi \nabla)]=[(\bar\p_E,0)]~.\]

At the other extreme, suppose $\Ecal$ is maximally unstable among bundles with  holomorphic connections.
Then (up to a twist) the successive quotients $\Ecal_j$ in the Harder-Narasimhan filtration of $\Ecal$ are line bundles of the form $\Ecal_j=K^{\frac{n+1}{2}-j}$, and $\nabla$ necessarily induces the tautological isomorphism $\Ecal_j\simeq \Ecal_{j+1}\otimes K$ (i.e.\ $(\Ecal,\nabla)$ is an \emph{oper}).
In this case the limit on the left hand side above is a Fuchsian point \eqref{eq VHS oper hitchin section}.

In general, however,  a  Higgs bundle induced from the Harder-Narasimhan filtration of $(\Ecal,\nabla)$ in this way will not be semistable.
Nevertheless, in  \cite{SimpsonDeRhamStrata} an iterative process is described which produces a semistable Higgs bundle, and identifies the limit above.

A  filtration $\Fcal$ by holomorphic subbundles $\Fcal_1\subset\dots\subset \Fcal_\ell=\Ee$ is called \emph{Griffiths transverse} if 
$\nabla(\Fcal_j)\subset \Fcal_{j+1}\otimes K$,
for all $1\leq j\leq \ell-1$.
Denote the associated graded of such a filtration by 
\[\Gr\Ecal=(E,\bar\p_{\mathrm{Gr} E})=\Ee_1\oplus \cdots\oplus \Ee_\ell,\] where $\Ee_j=\Fcal_j/\Fcal_{j-1}\ .$ 
The Griffiths transverse connection induces an $\Ocal_X$-linear map on the associated graded:
$$\nabla: \Ecal_j=\Fcal_j/\Fcal_{j-1}\to \Ecal_{j+1}\otimes K=(\Fcal_{j+1}/\Fcal_{j})\otimes K$$ which we denote by $\Phi_\Fcal$. 
Thus, associated to a Griffiths transverse filtration $\Fcal$, there is a Higgs bundle $(E,\bar\p_{\mathrm{Gr} E},\Phi_\Fcal)$, where 
$\Phi_\Fcal\in\bigoplus\limits_{j=1}^{\ell-1}H^0(\Hom(\Ee_j,\Ee_{j+1})\otimes K)$.
\begin{Lemma}\label{Lem identifying limit}
    If  $(E,\bar\p_{\mathrm{Gr} E},\Phi_\Fcal)$ is semistable then 
$$ \lim\limits_{\xi\to0}[(\xi,\bar\p_E,\xi \nabla)]=[(\bar\p_{\mathrm{Gr} E},\Phi_\Fcal)]\ .
$$
\end{Lemma}
\begin{proof}
Choose a smooth identification of $\Gr E$ with $E$ compatible with the obvious filtrations. With this choice, the bundle $\End(E)$ has a decomposition analogous to \eqref{eq End decomp}; namely,
\[\End(E)=\bigoplus\limits_{j=1-\ell}^{\ell-1}\End_j(E)~.\]
Note that the Higgs field $\Phi_\Fcal$ is a holomorphic section of $\End_{-1}(E)\otimes K.$  
Since  $\Fcal$ is Griffiths transverse there are $\beta_j\in\Omega^{0,1}(\End_j(E))$ and $\varphi_j\in\Omega^{1,0}(\End_j(E))$ so that
\[(\xi,\bar\p_E, \xi \nabla)=\Bigl(\xi~,~ \bar\p_{\mathrm{Gr} E}+\sum\limits_{j=1}^{\ell-1}\beta_j~,~ \xi \Phi_\Fcal+\xi D^{1,0}_{\mathrm{Gr} E}+\sum\limits_{j=1}^{\ell-1}\xi \varphi_j~\Bigr)\ ,\]
where is the $D^{1,0}_{\mathrm{Gr} E}$ connection induced by $\nabla.$
The gauge transformation $g_\xi$ from \eqref{eq fixed point hol gauge trans} acts as 
\[(\xi,\bar\p_E,\xi \nabla)\cdot g_\xi=\Bigl(\xi~,~\bar\p_{\mathrm{Gr} E}+\sum\limits_{j=1}^{\ell-1}\xi^j\beta_j~,~\Phi_{\Fcal}+\xi D^{1,0}_{\mathrm{Gr} E}+\sum\limits_{j=1}^{\ell-1}\xi^{j+1}\varphi_j\Bigr)~.\]
Now taking the limit $\xi\to 0$ yields $(\bar\p_{\mathrm{Gr} E},\Phi_\Fcal).$
\end{proof}

The key result is the following.
\begin{Theorem}[{\cite[Theorem 2.5]{SimpsonDeRhamStrata}}]
\label{THM Simpson transverse filtration}
  Given a $(\Ecal,\nabla)$ there exists a Griffiths transverse filtration $\Fcal$ of $\Ecal$ so that the associated Higgs bundle $(\bar\p_{\mathrm{Gr} E},\Phi_\Fcal)$ is semistable. Moreover, if $(\bar\p_{\mathrm{Gr} E},\Phi_\Fcal)$ is stable, then the  filtration is unique. 
\end{Theorem}
Recall that if $[(\bar\p_0,\Phi_0)]\in \VHS_\alpha$, then a holomorphic flat bundle $[(\Ecal, \nabla)]\in\Mm_{\rdR}$ is in  $\rW^1_\alpha(\bar\p_0,\Phi_0)$  if and only if 
\[\lim_{\xi\to0}[(\xi,\bar\p_E,\xi \nabla)]=[(\bar\p_0,\Phi_0)]~.\] 
The following is the de Rham analogue of Proposition \ref{prop stable man of VHS normal form}. It follows easily from Theorem \ref{THM Simpson transverse filtration}, Lemma \ref{Lem identifying limit}, and Proposition \ref{Prop flat conn of fixed point}.

\begin{Proposition} \label{prop:dr-normal-form}
Let $(\bar\p_0,\Phi_0)$ be a stable complex variation of Hodge structure with harmonic metric $h$. Let 
$(\Ecal,\nabla)$, $\Ecal=(E,\bar\p_0+\Phi^{\ast_h}_0)$, $\nabla=\p_0^h+\Phi_0$, be the associated
  holomorphic  bundle with holomorphic connection. Then $[(\Ecal,\nabla)]\in\rW^1_\alpha(\bar\p_0,\Phi_0)$  if and only if, after a gauge transformation,
\[\xymatrix{\bar\p_E-\bar\p_0-\Phi_0^{\ast_h}\in \Omega^{0,1}(N^+)&\text{and}&\nabla-\p_0^h-\Phi_0\in \Omega^{1,0}(L\oplus N^+)}~,\]
where $N_+$ and $L$ are defined in \eqref{eq N+ notation}.
\end{Proposition}

We now prove the analogous result to Proposition \ref{prop:good-gauge} in the de Rham picture.
\begin{Proposition} \label{prop:good-gauge2}
Let $(\bar\p_E,\Phi)$ be a stable complex variation of Hodge structure with harmonic metric $h$ and associated flat connection 
$
D=\dbar_E+\partial_E^h+\Phi+\Phi^{\ast_h}
$.
Suppose $(\beta,\varphi)\in 
\Omega^{0,1}(N_+)\oplus \Omega^{1,0}(L\oplus N_+)$ is such that $D+\beta+\varphi$ is flat.
 Then there is a unique smooth section $f\in \Omega^0(N_+)$ such that applying the complex gauge transformation $g=\Ibold+f$, 
$$
(D+\beta+\varphi)\cdot g=D+\widetilde\beta+\widetilde\varphi\ ,
$$
then $(\widetilde\beta, \widetilde\varphi)\in \Scal^+_{(\dbar_E,\Phi)}$. Moreover, if $(\beta_1,\varphi_1), (\beta_2,\varphi_2)\in \Scal^+_{(\dbar_E,\Phi)}$ are such that $$
(D+\beta_1+\varphi_1)\cdot g=D+\beta_2+\varphi_2\ 
$$
for some $g\in \Gcal(E)$,
then $g$ is in the center and $(\beta_1,\varphi_1)= (\beta_2,\varphi_2)$.
\end{Proposition}

\begin{proof}
 For the existence part, note that the only difference with the proof of Proposition  \ref{prop:good-gauge} is the addition of the term $g^{-1}\partial_0g$ for the change of $\varphi$ under a complex gauge transformation. But in terms of the graded pieces,  $[\Phi^\ast, g^{-1}\partial_0g]_j=0\, \mod j$, and so the same proof applies in this case as well. 

 We move on to prove the second statement.  
 Write $g=\Ibold+f$, a calculation  yields
$$
(D'')^\ast D''(f)=-i\Lambda\left([\Phi^\ast, f]\wedge\varphi_2+\varphi_1\wedge[\Phi^\ast, f]+\partial_0f\wedge\beta_2+\beta_1\wedge\partial_0 f\right)\ .
$$
But since $\Phi^\ast$ and $\beta_i$ raise the grading,  we have
$
(D'')^\ast D'' (f_j)=0\mod j
$.
By an inductive argument and the stability of $(\dbar_E,\Phi)$, the endomorphism $f$, and hence also $g$, must lie in the center. 
\end{proof}

\begin{Corollary} \label{cor:pdr}
Suppose $[(\dbar_E,\Phi)]\in \VHS_\alpha$ is stable. Then
the map $p_{\rdR}$ in \eqref{eqn:p} gives a biholomorphism $\Scal^+_{(\dbar_E,\Phi)}\isorightarrow \rW^1_\alpha(\dbar_E,\Phi)$.
\end{Corollary}

\begin{proof}  Immediate from Propositions \ref{prop:dr-normal-form}, \ref{prop:good-gauge2} and the fact that $p_\rdR$ is holomorphic.
\end{proof}
\begin{Remark}
     Note that Corollary \ref{cor:pdr} generalizes the parameterization of the components of opers from Example \ref{ex:fuchsian-slice} to the fibers of every smooth strata.
\end{Remark}

\subsection{Proofs}
We now put together the results from Sections \ref{sec: def theory section} and \ref{sec:strata} to prove Theorems \ref{thm:global-slice} and \ref{thm:transverse}, Proposition \ref{prop:lagrange},  and Corollary \ref{cor:ConformalLimit}.
\begin{proof}[Proof of Theorem \ref{thm:global-slice}]  
 Part (1) of Theorem \ref{thm:global-slice} is contained in Lemma \ref{lem:dimension} and Theorem \ref{thm:kuranishi}. 
 Part (2) follows from Corollary \ref{cor:ph}.
  For part (3), for each $\lambda\in \CBbb$ we extend the definition \eqref{eqn:p} to:
 \begin{align} \label{eqn:p-lambda}
 \begin{split}
 p_\lambda : \Scal_{(\dbar_E,\Phi)} &\lra \Mm_{\Hod} \\
 (\beta,\varphi)&\mapsto [(\lambda, \dbar_E+\lambda\Phi^\ast+\beta, \lambda\partial_E^h+\Phi+\varphi)]
 \end{split}
 \end{align}
 where $h$ is the harmonic metric for $(\dbar_E,\Phi)$. 
 The map $p_{\Hod}$ is then defined by $p_{\Hod}(\beta,\varphi,\lambda)=p_{\lambda}(\beta,\varphi)$. Now 
 $p_\lambda$ is a biholomorphism onto the fiber at $\lambda$. Indeed, this is true for $\lambda=0$ and $\lambda=1$. The
  $\CBbb^\ast$-action identifies the fibers $\lambda\neq 0$ with $\Mm_{\rdR}$. The result then follows by a small modification of Propositions \ref{prop:dr-normal-form} and \ref{prop:good-gauge2}.
\end{proof}

\begin{proof}[Proof of Proposition \ref{prop:lagrange}] For the first statement, 
since
$\rW^0_\alpha(\bar\p_0,\Phi_0)$ is  preserved by the complex structure $I$ and  is half-dimensional by Lemma \ref{Lemma dim of stable man}, it suffices to prove that
$\rW^0_\alpha(\bar\p_0,\Phi_0)$ is isotropic.  But by Corollary \ref{cor:ph}, tangent vectors to  $\rW^0_\alpha(\bar\p_0,\Phi_0)$ are represented by elements  in $\Omega^{0,1}(N_+)\oplus \Omega^{1,0}(L\oplus N_+)$.
The result now follows directly from the definition \eqref{eq hol symplectic form on Higgs bundles} of $\omega_I^\C$.
The second statement is similar: we note that 
$\rW^1_\alpha(\dbar_E,\Phi)$ is $J$-holomorphic and half-dimensional,
and the fact that it is isotropic  follows from Corollary \ref{cor:pdr} and the expression
 \eqref{eq: ABG hol symplectic form on de Rham}. 
    \end{proof}

\begin{proof}[Proof of Theorem \ref{thm:transverse}]
Let $h$ be the harmonic metric for $(\dbar_E,\Phi)$. By Proposition \ref{Prop first variation of metric}, the derivative of $h$ vanishes to first order for deformations in the Hodge slice. In particular, if $\ubold(t)\in \Scal^+_{(\dbar_E,\varphi)}$ is a smooth one (real) parameter family with $\ubold(0)=0$, $\dt\ubold(0)=(\beta,\varphi)\in\Hcal^1_+(\dbar_E,\Phi)$, then
$$
\Tcal[\ubold(t)]=[ \dbar_E+t\beta+\partial_E^h-t\beta^{\ast_h}+ \Phi+t\varphi+\Phi^{\ast_h}+t\varphi^{\ast_h}+O(t^2)]\ .
$$
Hence, in these coordinates, $d\Tcal(\beta,\varphi)=\beta+\varphi^\ast-\beta^\ast+\varphi$, where we have dropped $h$ from the notation. But according to \eqref{eq complex structures on de Rham}, 
$$
\beta+\varphi^\ast-\beta^\ast+\varphi = \mu - K(\mu)
$$
for the tangent vector $\mu=\beta+\varphi\in T_{\Tcal[(\dbar_E,\varphi)]}\rW^1_{\alpha}(\dbar_E,\Phi)$. 
This completes the proof.
\end{proof}

\section{Conformal Limits}
Let $[(\dbar_E,\Phi)]$ be a stable Higgs bundle with harmonic metric $h$, and
$$
\lim_{\xi\to 0}[(\dbar_E,\xi\Phi)]=[(\dbar_0,\Phi_0)]\in \VHS_\alpha\ .
$$
 Recall the real twistor line in $\Mm_{\Hod}$:
$$
\tau(\xi)=\dbar_E+\xi\partial_E^h+\Phi+\xi\Phi^\ast\ 
$$
which passes through $[(\dbar_E,\Phi)]$ at $\xi=0$ and $\Tcal([(\dbar_E,\Phi)])$ at $\xi=1$. In general,  $\tau(\xi)$ does not lie in the fiber $W_\alpha(\dbar_0,\Phi_0)$.
In other words, nonabelian Hodge does not map $W^0_\alpha(\dbar_0,\Phi_0)$ to $W^1_\alpha(\dbar_0,\Phi_0)$.
For example, if  $(\dbar_0,\Phi_0)$ is the Fuchsian point, then $\Tcal([(\dbar_E,\Phi)])$ is a real representation for every $[(\dbar_E,\Phi)]\in W^0_\alpha(\dbar_0,\Phi_0)$, whereas the holonomy of an oper in $W^1_\alpha(\dbar_0,\Phi_0)$ is not real in general (see Example \ref{ex:fuchsian-slice}). 

By contrast,
 the $\CBbb^\ast$-orbit of a point in $W^1_\alpha(\dbar_0,\Phi_0)$ does lie in $W_\alpha(\dbar_0,\Phi_0)$, but it collapses the entire stratum $W^1_\alpha(\dbar_0,\Phi_0)$ to the point $[(\dbar_0,\Phi_0)]$ at $\xi=0$. Notice, however, that these two holomorphic sections  of $\Mm_{\Hod}\to \CBbb$ -- the twistor lines and the $\CBbb^\ast$-orbits --   agree for the twistor line through $[(\dbar_0,\Phi_0)]$ itself. In Theorem \ref{thm:global-slice}, we have produced a third section, $p_\xi$ from \eqref{eqn:p-lambda}, which also agrees on the twistor lines through $\VHS$, but which gives a  biholomorphism 
 $W^0_\alpha(\dbar_0,\Phi_0)$ and $W^1_\alpha(\dbar_0,\Phi_0)$.

In  \cite{GaiottoConj} Gaiotto introduced a rescaled version of these two sections to produce the family \eqref{eqn:gaiotto}. In the case where $[(\dbar_0,\Phi_0)]$ is the Fuchsian point (Example \ref{ex:fuchsian}), he  conjectured the existence of a well-defined limit as $R\to 0$, and that it should be an oper. As evidence, he pointed to the analogy between the cyclic structure of the Higgs field and the jet bundle description of opers.
The conjecture was recently proven in \cite{GaiottoLimitsOPERS}.  

In this section, we prove that in fact the limits exist in much greater generality. As an example, at the opposite extreme of the original conjecture is the case where $(E,\dbar_0)$ is a stable bundle. It is not hard to see that then the harmonic metric for $(\dbar_E, R\Phi)$ converges to the Hermitian-Einstein metric $h_0$ on $(E,\dbar_0)$. Hence, the flat connection in \eqref{eqn:gaiotto} simply converges  to $\dbar_E+\partial_E^{h_0}+\hbar^{-1}\Phi$, and this in turn lies in the open stratum in $\Mm_{\rdR}$.

Let $[(\dbar_E,\Phi)]\in \VHS_\alpha$ be a stable Higgs bundle,
and let $h$ denote the harmonic metric for $(\dbar_E,\Phi)$. 
Throughout this section, an unannotated $\ast$ will refer to adjoints with respect to $h$, and we set $\partial_E:=\partial_E^h$.  Endomorphisms will be written with respect to the splitting in Proposition \ref{Prop: fixedpoints of C* action}.
By Corollary \ref{cor:ph}, any point in $\rW^0_\alpha(\dbar_E,\Phi)$ is of the form $[(\dbar_E+\beta,\Phi+\varphi)]$ for $\ubold=(\beta,\varphi)\in S^+_{(\dbar_E,\Phi)}$. Fix $\hbar>0$.
Recall from \eqref{eqn:p-lambda} that
$$
p_{\hbar}(\ubold)=\dbar_E+\hbar\Phi^\ast+ \beta +\hbar\partial_E^h+\Phi+\varphi\ .
$$
For convenience, set $\dbar_\ubold:=\dbar_E+\beta$ and $\Phi_\ubold:=\Phi+\varphi$.
For  $R>0$,the connection
$$
D_{(\ubold, R)}:= \hbar^{-1}\Phi_{\ubold}+ \dbar_\ubold+\partial^{\ast_{h(\ubold, R)}}_\ubold +\hbar R^2\Phi_{\ubold}^{\ast_{h(\ubold, R)}} 
$$
is flat for the harmonic metric $h(\ubold, R)$ for $(\dbar_\ubold, R\Phi_\ubold)$. 

The following is a  restatement of  Theorem \ref{thm:ConformalLimit} and Corollary \ref{cor:ConformalLimit}.
\begin{Proposition}
As $R\to 0$, the flat connections $D_{(\ubold, R)}$ converge smoothly to 
$$
 \dbar_E+\partial_E+ \beta+\hbar^{-1}\Phi_{\ubold}+\hbar\Phi^\ast\ .
$$
In particular,
$$\lim_{R\to 0} [D_{(\ubold, R)}]=
\hbar^{-1}\cdot p_{\hbar}(\ubold)\ .
$$
\end{Proposition}

\begin{proof} Since the argument is modeled closely on the proof in \break\cite{GaiottoLimitsOPERS}, we shall be brief. 
Let $[(\dbar_E,\Phi)]\in \VHS_\alpha$ be a stable point, 
First, some preliminaries. Let $g\in \Gcal(E)$. 
Define a new metric $g(h)$ by the rule:
$\langle u,v\rangle_{g(h)}=\langle gu, gv\rangle_{h}$.  Let $k=g^\ast g$. Then we have for any Higgs bundle $(\dbar_E,\Phi)$,
\begin{align}
\Phi^{\ast_{g(h)}} &= k^{-1}\Phi^{\ast}k \label{eqn:phi-change} \\
\partial^{g(h)}_E &= k^{-1}\circ \partial_E\circ k  \label{eqn:d-change} \\
F_{(\dbar_E, g(h))}&= F_{(\dbar_E, h)} +\dbar_E(k^{-1} \partial_E(k)) \label{eqn:f-change}
\end{align}
Recall that an expression of the form $F_{(\dbar_E, h')}$ denotes the curvature of the Chern connection associated to $\dbar_E$ and a metric $h'$.

With this understood, we first modify $h$ to get a metric $h_R'$ by using:
$$
g=\left(\begin{matrix} R^{m_1/2}&&0\\&\ddots&\\ 0&&R^{m_\ell/2}\end{matrix}\right)\ , \ k=\left(\begin{matrix} R^{m_1}&&0\\&\ddots&\\ 0&&R^{m_\ell}\end{matrix}\right)
$$
where  $m_j-m_{j+1}=2$, $j=1,\ldots, \ell-1$,  and $\sum_{i=1}^\ell (\rk E_j)m_j=0$.
Then using \eqref{eqn:phi-change} and \eqref{eqn:d-change}, we have:
\begin{align*}
\dbar_\ubold+\partial^{h_R'}_\ubold&= \dbar_E+\partial_E+ \beta-\sum_{j=1}^{\ell-1}R^{2j} (\beta_j)^\ast  \\
 \Phi^{\ast_{h_R'}}_\ubold&=R^{-2}\Phi^\ast+  \sum_{j=0}^{\ell-1}R^{2j} (\varphi_j)^\ast
\end{align*}
Taking the limit as $R\to 0$ we have
\begin{align}  \label{eqn:limit-flat}
\begin{split}
D_{(\ubold, 0)} :&=\lim_{R\to 0}\left\{ \hbar^{-1}\Phi_{\ubold}+ \dbar_\ubold+\partial^{h_R'}_\ubold+\hbar R^2\Phi_{\ubold}^{\ast_{h_R'}} \right\}   \\
&= \hbar^{-1}\Phi_{\ubold}+\dbar_E+\partial_E+ \beta+\hbar\Phi^\ast \ .
\end{split}
\end{align}

Let us record
\begin{align}
R^2 \Phi^{\ast_{h_R'}}_\ubold&=\Phi^\ast+  \sum_{j=0}^{\ell-1}R^{2j+2} (\varphi_j)^\ast \label{eqn:phi-prime}\\
F_{(\dbar_\ubold, h_R')}&=F_{(\dbar_E, h)}+\biggl[ \beta, \sum_{j=1}^{\ell-1}R^{2j} (\beta_j)^\ast\biggr]+
\partial_E\beta-\sum_{j=1}^{\ell-1}R^{2j} \dbar_E(\beta_j)^\ast \label{eqn:curv-prime}
 \end{align}
The goal now is to find a solution of the form $h(\ubold, R)=g_R(h_R')$, where $g_R=\exp(f_R)$, for $f_R$ a (traceless) $h_R'$-hermitian endomorphism. 
Then $k_R=\exp(2f_R)$ in eqs.\ \eqref{eqn:phi-change}--\eqref{eqn:f-change}.  This means we set $N_{(\ubold, R)}(f_R)=0$, where
\begin{equation*}\label{eqn:N}
N_{(\ubold, R)}(f_R):=
\dbar_\ubold(e^{-2f_R}\partial_\ubold^{h_R'}(e^{2f_R}))+ F_{(\dbar_\ubold, h_R')}+\left[\Phi_\ubold, e^{-2f_R}\left(R^2\Phi^{\ast_{h_R'}}_\ubold \right)e^{2f_R}\right]
\end{equation*}
Note that: $N_{(\ubold, 0)}(0)=0$.
  Now consider the linearization in $R$ at $R=0$.  We denote derivatives at $R=0$ by $\dt f$, for example.  Now $f_R$ is determined by a section $(f_L(R), f_+(R))$ of $L\oplus N_+$. 
  Specifically, 
  $$
f_R=f_L(R)+ f_{+}(R)+\sum_{j=1}^{\ell-1} R^{2j} \left(f_{+, j}(R)\right)^\ast\ .
$$
Hence, we may consider $N_{(\ubold, R)}$ as a smooth map $C^{2,\alpha}(L\oplus N_+)\to C^\alpha(L\oplus N_+)$.   By \eqref{eqn:phi-prime} and \eqref{eqn:curv-prime}, 
\begin{equation} \label{eqn:dN}
dN_{(\ubold, 0)}(0)\dt f=\dbar_\ubold\partial_E \dt f+\left[\Phi_\ubold, \left[\Phi^\ast, \dt f\right]\right]\ .
\end{equation}

\begin{Lemma} \label{lem:no-kernel}
 $\ker dN_{(\ubold, 0)}(0)=\{0\}$.
\end{Lemma}

\begin{proof}
Suppose $\dt f\in \ker dN_{(\ubold, 0)}(0)$. Write the decomposition into components as $\dt f=\dt f_L+\dt f_+$.
Using \eqref{eqn:hitchin}, 
$$
\dbar_E\partial_E=-\partial_E\dbar_E+F_{(\dbar_E, h)}=-\partial_E\dbar_E-\Phi\Phi^\ast-\Phi^\ast\Phi\ .
$$
  Then a direct calculation from \eqref{eqn:dN} yields
$
(D'')^\ast D'' ( \dt f_L)=
0$. As in the proof of Proposition \ref{prop:good-gauge}, the operator on the left is invertible, so $\dt f_L=0$.  Using this, it follows from \eqref{eqn:dN} that the first graded piece of $\dt f_+$ satisfies $(D'')^\ast D'' ({\dt f}_{+,1})=
0$, and therefore vanishes as well. More generally, if $\dt f_{+,k}=$ for $k<j$ then we have $(D'')^\ast D'' (\dt f_{+,j})=
0$. The result now follows.
\end{proof}
\noindent
By the lemma and the implicit function theorem, there is a unique  $f_R$, $f_0=0$, such that $N_{(\ubold, R)}(f_R)=0$  for $R$ sufficiently small. 
As in \cite{GaiottoLimitsOPERS}, and by \eqref{eqn:limit-flat}, this completes the proof. 
\end{proof}

\begin{Remark}  \label{rmk:discontinuous}
Let us point out that the conformal limit, regarded as a map from an open dense set in $\Mm_{\rH}$ to $\Mm_{\rdR}$, is not continuous. For example, if it were  then  the image of the $\CBbb^\ast$-orbit of a Higgs bundle $(\dbar_E,\Phi)$ with $(E,\dbar_E)$ stable and $\Phi\neq 0$ nilpotent, would complete to a holomorphically embedded $\PBbb^1\subset \Mm_{\rdR}$ which, since the latter is an affine variety, cannot exist.
\end{Remark}

\bibliography{mybib}{}
\bibliographystyle{alpha}

\end{document}